\documentclass[a4paper,12pt,twoside]{amsart}

 \usepackage{amsfonts,amssymb,amscd}
\usepackage{graphicx}

\usepackage[left=2.3cm,top=3.3cm,right=2.3cm,bottom=2.7cm]{geometry}
\usepackage{mathrsfs}
\usepackage{xfrac}
\let\mathcal\mathscr
\usepackage[colorlinks=true,citecolor=blue, urlcolor=blue, breaklinks, linkcolor=black]{hyperref}

\usepackage[all,ps,cmtip]{xy}

\DeclareRobustCommand{\SkipTocEntry}[5]{}

\def\Z{{\bf Z}}

\def\C{{\bf C}}

\def\Q{{\bf Q}}
\def\P{{\bf P}}

\def\k{{\bf k}}

\makeatletter
\newcommand{\addresseshere}{%
\enddoc@text\let\enddoc@text\relax
}
\makeatother

\def\av{abelian variety}

\def\lra{\longrightarrow}
\def\lraa{\lra\hskip-6.05mm \lra}
\def\llra{\hbox to 10mm{\rightarrowfill}}

\def\lllra{\hbox to 15mm{\rightarrowfill}}

\newcommand{\longtwoheadrightarrow}{} 
\DeclareRobustCommand{\longtwoheadrightarrow}{\relbar\joinrel\twoheadrightarrow}

\def\lthra{\longtwoheadrightarrow}

\def\PA{{\widehat A}}
\def\PB{{\widehat B}}
\def\PK{{\widehat K}}

\def\PE{{\widehat E}}

\def\phi{{\varphi}}

\def\wX{{\widetilde X}}

\def\wY{{\widetilde Y}}
\def\wA{{\widetilde A}}

\def\cF{\mathcal{F}}

 \def\cL{{L}}

\def\cO{\mathcal{O}}

\def\dra{\dashrightarrow}
\def\tto{\twoheadrightarrow}
\def\isom{\simeq}

\def\k{\mathbf k}

\DeclareMathOperator{\rank}{rank}

\DeclareMathOperator{\Ker}{Ker}

\DeclareMathOperator{\codim}{codim}
\DeclareMathOperator{\Pic}{Pic}

\DeclareMathOperator{\id}{id}

\DeclareMathOperator{\Gr}{Gr}

\DeclareMathOperator{\isomlra}{\stackrel{{}_{\scriptstyle\sim}}{\lra}}
\DeclareMathOperator{\isomdra}{\stackrel{{}_{\scriptstyle\sim}}{\dra}}

\newtheorem{lemm}{Lemma}[section]
\newtheorem{theo}[lemm]{Theorem}
\newtheorem{coro}[lemm]{Corollary}
\newtheorem{prop}[lemm]{Proposition}

\newtheorem{theAlph}{Theorem}

\newtheorem{theoA}{Theorem}

\newtheorem{corA}[theoA]{Corollary}

\theoremstyle{definition}
\newtheorem{defi}[lemm]{Definition}
\newtheorem{rema}[lemm]{Remark}

\newtheorem{exam}[lemm]{Example}

\newtheorem{qu}[lemm]{Question}

\theoremstyle{remark}
\newtheorem*{remark*}{Remark}
\newtheorem*{note*}{Note}

\newcommand{\set}[1]{\left\{#1\right\}}

 \newcommand{\res}[2]{\left.#1\right|_{#2}}

\def\bw#1#2{\textstyle{\bigwedge\hskip-0.9mm^{#1}}\hskip0.2mm{#2}}

\author[O.~Debarre]{Olivier Debarre}
\address{D\'epartement Math\'ematiques et Applications\\UMR CNRS 8553\\PSL Research University\\\'Ecole Normale Su\-p\'e\-rieu\-re\\45 rue d'Ulm, 75230 Paris cedex 05, France}
\email{{olivier.debarre@ens.fr}}
\urladdr{\url{http://www.math.ens.fr/~debarre/}}

\author[Z.~Jiang]{Zhi Jiang}
\address{D\'epartement de Math\'ematiques d'Orsay\\UMR CNRS 8628\\Universit\'{e} Paris-Sud\\B\^{a}timent 425, F-91405 Orsay, France}
\email{Zhi.Jiang@math.u-psud.fr}
\urladdr{\url{http://www.math.u-psud.fr/~jiang/}}

\author[M.~Lahoz]{Mart\'{\i} Lahoz}
\address{Institut de Math\'{e}matiques de Jussieu -- Paris Rive Gauche\\UMR CNRS 7586\\Universit\'{e} Paris--Diderot\\B\^{a}timent Sophie Germain, Case 7012, 75205 Paris cedex 13, France}
\email{marti.lahoz@imj-prg.fr}
\urladdr{\url{http://webusers.imj-prg.fr/~marti.lahoz/}}

\makeatletter
\let\@wraptoccontribs\wraptoccontribs
\makeatother

\contrib[\qquad\qquad\qquad \lowercase{with an appendix by}]{William F. Sawin}

\thanks{M.~L.~is partially supported by MTM2012-38122-C03-02.}



\begin{document}
\title{Rational cohomology tori}

\begin{abstract}
We study normal compact K\"ahler spaces whose rational cohomology ring is isomorphic to that of a complex torus.
We call them rational cohomology tori.
First, we classify, up to dimension three, those with rational singularities.
Then, we give constraints on the degree of the Albanese morphism and the number of simple factors of the Albanese variety for rational cohomology tori of general type (hence projective) with rational singularities.
Their properties are related to the birational geometry of smooth projective varieties of general type, maximal Albanese dimension, and with vanishing holomorphic Euler characteristic.
We finish with the construction of  series of examples.\end{abstract}

\maketitle
 
\section*{Introduction}
\parskip=1mm

Given a compact complex manifold, one fundamental problem is to determine how much information is encoded in its underlying topological space.

Kodaira and Hirzebruch proved that any compact K\"ahler manifold of {\em odd} dimension $n$ which is homeomorphic to $\P^n$ is actually isomorphic to $\P^n$ (\cite{hk}, \cite[Theorem 1]{mor}).
A stronger property is actually conjectured:
it should be sufficient to assume that the rings $H^\bullet(X,\Z)$ and $H^\bullet(\P^n,\Z)=\Z[x]/(x^{n+1})$ are isomorphic and $c_1(T_X)>0$ to deduce that $X$ is biholomorphic to $\P^n$
(this is known in dimensions $\le 5$; \cite[Theorem 1]{fujitaPn}, \cite{vandeVen}).

In \cite[Theorem 70]{cat} (see also Theorem \ref{thm:cat}), Catanese observed that complex tori $X$ satisfy this stronger property: they can be characterized among compact K\"ahler manifolds by the fact that the cup product induces a ring isomorphism
\begin{equation}\label{eqZ}
\bw{\bullet}{H^1(X,\Z)} \isomlra H^\bullet(X,\Z).
\end{equation}
The K\"ahler assumption is essential since one can construct non-K\"ahler compact complex manifolds (hence not biholomorphic to complex tori)
 that satisfy \eqref{eqZ} (see Example \ref{exam:nonKaehler}).
If we replace \eqref{eqZ} by an isomorphism
\begin{equation}\label{eqQ}
\bw{\bullet}{H^1(X,\Q)} \isomlra H^\bullet(X,\Q)
\end{equation}
of $\Q$-algebras, Catanese asked whether this property still characterizes complex tori (\cite[Conjecture 71]{cat}).

We say that a normal compact K\"ahler space $X$ is a {\em rational cohomology torus} if it satisfies \eqref{eqQ}.
The main objective of this article is to study the geometry of these spaces: give restrictive properties and construct examples (some smooth) which are not complex tori, thereby answering negatively Catanese's question.

The Albanese morphism of a smooth rational cohomology torus $X$ is
finite (\cite[Remark 72]{cat}).
A result of Kawamata (Remark \ref{rmk:algebraic}) then says that there is a morphism $X\to Y$ which is an Iitaka fibration for $X$, such that $Y$ is algebraic and has again a finite morphism to a torus.
We prove that $Y$ is also a rational cohomology torus, but possibly singular.
This leads to the following result.

\begin{theAlph}\label{thm:Iitaka}
Let $f\colon  X\rightarrow A$ be a finite morphism from a normal compact K\"ahler space $X$ to a torus $A$.
Consider the sequence of Iitaka fibrations
\begin{equation*}
X\xrightarrow{\ I_X\ } X_1\xrightarrow{\ I_{X_1}\ }X_2\lra\cdots \lra X_{k-1}\xrightarrow{\ I_{X_{k-1}}\ } X_{k}.
\end{equation*}
Then $X$ is a rational cohomology torus if, and only if, $X_k$ is a rational cohomology torus.
Moreover,
\begin{itemize}
\item either $X_k$ is a point and we say that $X$ is an {\em Iitaka torus tower};
\item or $X_k$ is of general type  (of positive dimension).
\end{itemize}
\end{theAlph}

It is easy to construct smooth projective surfaces which are Iitaka torus towers but not complex tori (Example \ref{exam:surface}).
Since the product of two rational cohomology tori is again a rational cohomology torus, this already gives a negative answer to Catanese's question in any dimension $\ge 2$.

The next question is whether all rational cohomology tori are Iitaka torus towers.
By Theorem \ref{thm:Iitaka}, this is the same as asking whether there exist (possibly singular) projective rational cohomology tori of general type.
This reduces our problem to the algebraic category; however, the price we have to pay is that we need to deal with singular varieties.

Rational singularities turn out to be the suitable kind of singularities to work with, because they are stable under the construction of the sequence of Iitaka fibrations of Theorem~\ref{thm:Iitaka}.
Moreover, any desingularization of a projective rational cohomology torus of general type with rational singularities has maximal Albanese dimension and vanishing holomorphic Euler characteristic (Proposition \ref{cor:pg1}).
These varieties were studied in \cite{CDJ,CJ}.
Building on these results,
we give a classification,  in dimensions up to three, of rational cohomology tori with rational singularities.

\begin{theAlph}\label{thm:dim2-3} Let $X$ be a compact K\"ahler space with rational singularities.
\begin{itemize}
\item[1)] If $X$ is a surface, $X$ is a rational cohomology torus if, and only if, $X$ is an Iitaka torus tower.
\item[2)] If $X$ is a threefold, $X$ is a rational cohomology torus if, and only if,
\begin{itemize}
\item either $X$ is an Iitaka torus tower;
\item or $X$ has an \'etale cover which is a Chen--Hacon threefold ($X$ is then singular, of general type).
\end{itemize}
\item[3)]
Starting from dimension $4$, there exist smooth rational cohomology tori  that  are not Iitaka torus towers.
\end{itemize}
\end{theAlph}

Chen--Hacon threefolds were constructed in \cite[\S4, Example]{ch-Iitaka} (see also Example \ref{exam:el&ch}).
The $n$-folds we construct for item 3) have Kodaira dimension  any number in   $\set{3,\dots,n-1}$.
In Corollary \ref{cor:smoothrtgt} of the Appendix, William Sawin shows that {\em smooth} rational cohomology tori of general type do not exist (but we construct in Example \ref{exam-dim}  {\em singular} rational cohomology tori of general type in any dimension $\ge3$).

After this classification result,
we focus on giving restrictions  on 
rational cohomology tori of general type.

\begin{theAlph}\label{thm:deg-rct}
Let $X$ be a projective variety of general type with rational singularities.
Assume that $X$ is a rational cohomology torus, with Albanese morphism $a_X$.
There exists a prime number $p $ such that $p^2\mid \deg(a_X)$.
\end{theAlph}

\noindent Moreover, if $\deg (a_X)=p^2$, the morphism $a_X$ is a $( \Z/p\Z)^2$-cover of its image (Corollary \ref{cor:extremal-rct}).

We deduce Theorem \ref{thm:deg-rct} and Corollary \ref{cor:extremal-rct} from the analogous restrictions on the degree of the Albanese morphism of a smooth projective variety $X$ of general type, maximal Albanese dimension, and $\chi(X,\omega_X)=0$ (Theorems \ref{thm:deg-chi0} and \ref{thm:extremal-chi}).
Theorem \ref{thm:deg-chi0} is probably the deepest result of this paper and a key step in its proof is the description of {\em minimal primitive varieties of general type with $\chi=0$} (Theorem \ref{thm:minimal}): 
{\em primitive} in the sense that there exist no proper subvarieties with $\chi=0$ through a general point, and {\em minimal} with respect to the degree of birational factorizations of the Albanese morphism (Definition \ref{def:primitive}).
The most difficult part of the proof of Theorem \ref{thm:deg-chi0} is to realize these varieties as Galois quotients of products of lower-dimensional varieties (Lemma \ref{lem:str}).

Continuing with the idea of giving constraints to the existence of rational cohomology tori of general type, we prove the following condition on the number of simple factors of their Albanese varieties.

\begin{theAlph}\label{thm:factors-rct}
Let $X$ be a projective variety of general type with rational singularities.
Assume that $X$ is a rational cohomology torus, with Albanese morphism $a_X\colon X\to A_X$, and let $p $ be the smallest prime divisor of $\deg (a_X)$.
Then $A_X$ has at least $p+1$ simple factors.
\end{theAlph}

Section \ref{sec:examples} is devoted to the construction of examples.
First, we construct in Example \ref{exam:chi0-p2}, for each prime $p$,
minimal primitive varieties $X$ of general type  of dimension $p+1$  with $\chi(X,\omega_X)=0$ whose Albanese morphisms are surjective $( \Z/p\Z)^2$-covers of a product of $p+1$ elliptic curves; they are finite Galois quotients of a product of $p+1$ curves.
These examples show that the structure of primitive varieties with $\chi=0$ is much more complicated than expected in \cite{CJ} (see Remark \ref{rmk:exam2}).

We then use techniques from \cite{par} to produce (singular) rational cohomology tori of general type  in 
any dimensions $\ge 3$
(Examples \ref{exam:rt-p2} and \ref{exam-dim}).
The first of these examples shows that
the lower bound on the degree of the Albanese morphism in Theorem \ref{thm:deg-rct} is optimal, and so is the lower bound on the number of factors of the Albanese variety in Theorem \ref{thm:factors-rct}.

\smallskip

The article ends with a series of examples of non-minimal primitive fourfolds with $\chi = 0$ whose Albanese variety has $4$ simple factors and whose Albanese morphism has degree $8$.

\addtocontents{toc}{\SkipTocEntry}
\subsection*{Notation}
We work over the complex numbers.
A variety is of general type if a (hence any) desingularization is of general type.
Given a compact K\"ahler manifold $X$, we denote by $A_X$ its Albanese torus and by $a_X\colon  X\rightarrow A_X$ its Albanese morphism (see Remark \ref{rmk:albanese}).
Given a complex torus $A$, we   denote by $\PA$ its dual.
A proper morphism $ X\to Y$ between normal varieties is a {\em fibration} if it is surjective with general connected fibers; it is {\em birationally isotrivial} if the general fibers are all birationally isomorphic to a fixed variety $F$.
By \cite{bbb}, this is equivalent to saying that after a finite (Galois) base change $Y'\to Y$, the product $X\times_YY' $ is birationally isomorphic, over $Y'$, to $F\times Y'$.

\addtocontents{toc}{\SkipTocEntry}
\subsection*{Acknowledgements}
We thank Fabrizio Catanese for explaining \cite[Conjecture 72]{cat} during a talk at the IMJ-PRG--ENS algebraic geometry seminar.
This was the starting point of this article.
It is also a pleasure to thank Lie Fu, S\'andor Kov\'acs, Stefan Schreieder, and Claire Voisin for useful conversations and comments.
M.~L.~worked on this article during his stay at IMPA  (Brazil)  and he is grateful for the support he received on this occasion.


\parskip=0mm
\section{Catanese's theorem and question}

In \cite{cat}, Catanese proved the following topological characterization of complex tori.

\begin{theo}[{\cite[Theorem 70]{cat}}]\label{thm:cat}
A compact K\"{a}hler manifold $X$ is
biholomorphic to a complex torus if, and only if, the cup product induces a ring isomorphism
\begin{equation}\label{eq2Z}
\bw{\bullet}{H^1(X, \Z)}\isomlra H^\bullet(X, \Z).
\end{equation}
\end{theo}

Let us recall the proof.
The Albanese map $a_X\colon  X\rightarrow A_X$ induces an isomorphism $a_X^{*1}\colon H^1(A_X, \Z)\isomlra H^1(X, \Z)$, and \eqref{eq2Z} then implies that $a_X^*\colon H^\bullet(A_X, \Z)\to H^\bullet(X,\Z)$ is also an isomorphism.
Set $n:=\dim (X)$;  the fact that $a_X^{*n}\colon H^{2n}(X, \Z)\isomlra H^{2n} (A_X, \Z) $ is an isomorphism implies that  $a_X$ is birational.
Moreover, since we have an isomorphism of the whole cohomology rings and $X$ is K\"ahler, $a_X$ cannot contract any subvariety of $X$ and is therefore finite.
Thus, $a_X$ is an isomorphism.

If we replace \eqref{eq2Z} by an isomorphism at the level of rational cohomology,
Catanese already observed that the Albanese morphism is still surjective and finite (\cite[Remark 72]{cat}).

\begin{defi}\label{def:rctorus}
Let $X$ be a normal compact K\"ahler space.
We say that $X$ is a {\em rational cohomology torus} if the cup product induces an isomorphism of $\Q$-algebras
\begin{equation}\tag{$*_X$}
\bw{\bullet}{H^1(X, \Q)} \isomlra H^\bullet(X, \Q).\label{*X}
\end{equation}
\end{defi}

We also give an a priori slightly different definition.

\begin{defi}\label{def:f-torus}
Let $X$ be a normal compact K\"ahler space and let $f\colon  X\rightarrow A$ be a morphism to a torus.
We say that $X$ is an {\em$f$-rational cohomology torus} if $f$ induces an isomorphism of $\Q$-algebras
\begin{equation}\tag{$*_f$}\label{*f}
f^*\colon  H^\bullet(A, \Q) \isomlra H^\bullet(X, \Q).
\end{equation}
\end{defi}

Condition \eqref{*f} certainly implies condition \eqref{*X}.
But because of singularities, the converse is not clear.
However, we show that Definitions \ref{def:rctorus} and \ref{def:f-torus} are equivalent.

\begin{prop}\label{prop:albanese}
Let $X$ be a normal compact K\"ahler space.
Assume that $X$ is a rational cohomology torus.
There exists a finite morphism $f\colon  X\rightarrow A$ onto a complex torus such that $X$ is an $f$-rational cohomology torus.
In particular,  the Deligne  Hodge structures on $H^\bullet(X) $ are pure.
\end{prop}

\begin{proof}
Set $n:=\dim(X)$.
There is a functorial mixed $\Q$-Hodge structure on $H^\bullet(X)$ (\cite{del2,del3,hironaka}, \cite[Proposition (1.4.1)]{fujiki}) for which the
wedge product $\bw{2n}{H^1(X)}\rightarrow H^{2n}(X)$ is a morphism of mixed Hodge structures.
Since $X$ is compact, $H^{2n}(X)$ is a one-dimensional pure Hodge structure, hence $h^1(X)=2n$, $W_0 H^1(X)=0$, and $H^1(X)$ carries a pure Hodge structure.
Therefore, $H^k(X)\isom \bw{k}{H^1(X)}$ has a pure Hodge structure for each $k\in\{0,\dots, 2n\}$.

Let $\mu\colon  X'\rightarrow X$ be a resolution of singularities with $X'$ K\"ahler.
Since $H^1(X)$ carries a pure Hodge structure, the pull-back map $\mu^*\colon  H^1(X)\rightarrow H^1(X')$ is injective.
Considering the Albanese morphism $a_{X'}\colon  X'\rightarrow A_{X'}$, we note that $a_{X'}^*\colon  H^1(A_{X'}, \Q)\rightarrow H^1(X', \Q)$ is an isomorphism.
Thus $H^1(A_{X'})$ has a sub-Hodge structure which is isomorphic to $H^1(X)$.
Moreover, there exists a quotient $\pi\colon A_{X'}\rightarrow A$ of complex tori such that $g^*H^1(A)=\mu^*H^1(X)$ as sub-Hodge structures of $H^1(X')$, where $g:=\pi a_{X'}\colon  X'\rightarrow A$.
We then have $\mu^*H^k(X, \Q)=g^*H^k(A, \Q)$ as subspaces of $H^k(X', \Q)$.

We claim that $g$ contracts every fiber of $\mu$.
Otherwise, let $F$ be an irreducible closed subvariety contained in a fiber of $\mu$ and assume $\dim \bigl(g(F)\bigr)>0$.
Let $\rho\colon  F'\rightarrow F$ be a desingularization with $F'$ K\"ahler and let $t\colon  F'\rightarrow F \hookrightarrow X'$ be the composition.
Let $\omega\in g^*H^2(A, \C)\subset H^2(X', \C)$ be the pullback of a K\"ahler form on $A$.
Since $\dim\bigl(g t (F')\bigr)>0$, the form $t^*\omega_X\in H^2(F', \C)$ is non-trivial.
This is a contradiction, since $\omega\in \mu^*H^2(X, \C)$ and $ \mu t (F')$ is a point.

Therefore, $g$ contracts every fiber of $\mu$.
Moreover, $\mu$ is birational and $X$ is normal, hence $\mu_*\cO_{X'}=\cO_X$.
Thus the morphism $g\colon  X'\rightarrow A$ factors through a morphism $f\colon  X\rightarrow A$ and $X$ is an $f$-rational cohomology torus.
\end{proof}

\begin{rema}\label{rmk:albanese}
If $X$ is a rational cohomology torus, Proposition \ref{prop:albanese} provides a finite morphism $f\colon  X\rightarrow A$ which has a universal property (up to translations) for morphisms from $X$ to complex tori.
We will call $f$ the \emph{Albanese morphism} and $A$ the \emph{Albanese variety} of $X$.
\end{rema}

The hypothesis $X$ K\"ahler in Theorem \ref{thm:cat} is essential:
the topological characterization of tori is not true without this hypothesis, as we show in the following example whose origins can be traced to \cite[p.\ 163]{bla} (see also \cite[Example 5.1]{ueno2}).

\begin{exam} [A non-K\"ahler integral cohomology torus] \label{exam:nonKaehler}
Let $E$ be an elliptic curve,
let $\cL$ be a very ample line bundle on $E$,
and let $\varphi$ and $\psi$ be holomorphic sections of $\cL$ with no common zeroes on $E$.
Set
\[
J_1:=\left(\begin{smallmatrix}
\;1\;&\;0\;\\&\\\;0\;&\;1\;
\end{smallmatrix}\right),\qquad
J_2:=\left(\begin{smallmatrix}
\;0\;&\;1\;\\&\\-1&\;0\;
\end{smallmatrix}\right),\qquad
J_3:=\left(\begin{smallmatrix}
0&\sqrt{-1}\\\sqrt{-1}&0
\end{smallmatrix}\right),\qquad
J_4:=\left(\begin{smallmatrix}
\sqrt{-1}&0\\0&-\sqrt{-1}
\end{smallmatrix}\right).
\]
Since $\det\big(\sum_{i=1}^4 \lambda_i J_i\big)=\sum_{i=1}^4\lambda_i^2$, the group
\[\Gamma:=\sum_{i=1}^4\Z J_i\left(\begin{matrix}
\varphi\\ \psi
\end{matrix}\right)
\]
is a relative lattice in the rank-$2$ vector bundle $\mathcal{V}:=\cL\oplus\cL$, over $E$.
The quotient $M:=\mathcal{V}/\Gamma$ is a complex manifold with a surjective holomorphic map $\pi\colon M\to E$.
By construction, $\pi$ is smooth, each fiber of $\pi$ is a complex $2$-dimensional torus, and its relative canonical bundle is
$
\omega_{M/E}=\pi^*\cL^{-2}$. 

One checks that $M$ is diffeomorphic to a real torus, hence $H^\bullet(M, \Z)=\bw{\bullet}{ H^1(M,\Z)}$,
but $M$ is not a complex torus, since it is not K\"ahler:
if it were, $\pi_*\omega_{M/E}$ would be semipositive (\cite[Theorem (2.7)]{fujita}).
\end{exam}

Answering a question  of 
 Schreieder, we would like to point out that the hypothesis that $\bw{\bullet}{H^1(X, \Q)} \isom H^\bullet(X, \Q)$ is a ring isomorphism is crucial 
to get a finite morphism to an abelian variety:
if we only assume that the Hodge numbers of $X$ are  those of a torus, the Albanese morphism is not even necessarily finite as we show in the following example.

\begin{exam}[A surface with the Hodge numbers of a torus but non-finite Albanese map]
 Let $\rho \colon D\to C$ be a double \'etale cover of smooth projective curves, where $C$ has genus $g\geq 2$, and let $\tau$ be the associated involution of $D$.
 Let $E$ be an elliptic curve and let $\sigma$ be the involution of $E$ given by multiplication by $-1$, with quotient morphism $\rho'\colon E\to \P^1$.
 Let $S:=(D\times E)/ \langle \tau \times \sigma\rangle$  be the diagonal quotient.
The surface $S$ is smooth and  its Hodge numbers  are
\[
\begin{array}{ccccc}
	&&1\\
	&2&&2&\\
	1&&4&&1
\end{array}
\]
Indeed,  we have 
$$
\begin{array}{rcll}
	 \rho_*\cO_D&=&\cO_C\oplus \cL^{-1}, &\text{where } \cL\in \Pic^0(C) \text{ is a 2-torsion point,}\\
	  \rho'_*\cO_E&=&\cO_{\P^1}\oplus \cO_{\P^1}(-2).&
\end{array}
$$
If we denote by $g\colon S\to C\times \P^1$ the natural double covering,  we have
\begin{eqnarray*}
	g_*\cO_S&=&\cO_{C\times \P^1}\oplus \bigl(\cL^{-1}\boxtimes\cO_{\P^1}(-2)\bigr),\\
	g_*\Omega^1_S&=& \bigl(\omega_C\boxtimes\cO_{\P^1} \bigr)\oplus  \bigl(\omega_C\otimes \cL)\boxtimes \cO_{\P^1}(-2) \bigr)\oplus\bigl(\cO_C\boxtimes\omega_{\P^1}\bigr)\oplus \bigl(\cL^{-1}\boxtimes\cO_{\P^1}\bigr), \\
	g_*\omega_S&=&\omega_{C\times \P^1}\oplus \bigl((\omega_{C}\otimes\cL) \boxtimes\cO_{\P^1}\bigr).  
\end{eqnarray*}
Note that the Albanese variety of $S$ and the Jacobian $J(C)$ are isogenous and that the Albanese morphism of $S$ contracts the elliptic curves $E$ that are the fibers of $S\to C$.
\end{exam}

\subsection{Iitaka torus towers}
By Proposition \ref{prop:albanese}, studying rational cohomology tori is equivalent to studying $f$-rational cohomology tori.
For an $f$-rational cohomology torus $X$, the property \eqref{*f} implies, since $X$ is K\"ahler, that $f$ is finite and surjective.
A theorem of Kawamata describes Iitaka fibrations for varieties with a finite morphism to a torus.

\begin{rema}[Reduction to algebraic varieties]\label{rmk:algebraic}
Let $X$ be a normal compact complex space and let $f\colon  X \rightarrow A$ be a finite morphism to a torus.
By \cite[Theorem 23]{kaw}, there are
\begin{itemize}
\item an abelian Galois \'etale cover $\pi\colon  \wX\rightarrow X$ with group $G$ induced by an \'etale cover of $A$,
\item a subtorus $K$ of $A$,
\item a normal projective variety $\hat{Y}$ of general type,
\end{itemize}
and a commutative diagram
\begin{equation}\label{eqn:diag1}
\xymatrix{
\wX\ar[r]^{\pi}\ar[d]_{I_{\wX}} & X\ar[r]^{f}\ar[d]_{I_X} & A\ar[d] \\
\hat{Y}\ar[r] & Y \ar[r]^{f_Y} & A/K }
\end{equation}
with the following properties:
\begin{itemize}
\item $I_X$ is the Stein factorization of the composition $X\rightarrow A\rightarrow A/K$ and is an Iitaka fibration of $X$;
\item $I_{\wX}$ is an analytic fibre bundle with fiber an \'etale cover $\widetilde{K}$ of $K$, and is also the Stein factorization of $I_X\pi\colon \wX\rightarrow Y$.
Hence, there is a natural $G$-action on $\hat{Y}$ which may not be faithful,
$I_{\wX}$ is $G$-equivariant, and $Y=\hat{Y}/G$.
\end{itemize}

For any finite group $G$ acting on a irreducible projective variety $V$, we have $H^\bullet(V, \Q)^G=H^\bullet(V/G, \Q)$ (\cite[Chapter III, Theorem 7.2]{bre}).
Thus,
\begin{eqnarray*}
H^\bullet(X, \Q)&=& H^\bullet(\wX, \Q)^G\\\nonumber
&=& \bigl(H^\bullet(\hat{Y}, \Q)\otimes H^\bullet(\widetilde{K}, \Q)\bigr)^G\\\nonumber
&=& H^\bullet(\hat{Y}, \Q)^G\otimes H^\bullet(K, \Q)\\\nonumber
&=& H^\bullet(Y, \Q)\otimes H^\bullet(K, \Q),
\end{eqnarray*}
where the second equality holds by the Leray--Hirsch theorem and the third equality holds because $G$ acts trivially on $H^\bullet(\widetilde{K}, \Q)$,
which is isomorphic to $H^\bullet(K, \Q)$.
Thus,
$(*_f)$ holds if, and only if, $(*_{f_Y})$ holds.
In particular, this allows to reduce the study of property $(*_f)$ to algebraic varieties.
\end{rema}

The following lemma, which implies Theorem \ref{thm:Iitaka} in the introduction, is an easy consequence of the previous remark.

\begin{lemm}\label{lem:char}
Let $f\colon  X\rightarrow A$ be a finite morphism from a normal compact complex space to a torus.
Let
\begin{equation}\label{eqn:ITT}
X\xrightarrow{\ I_X\ } X_1\xrightarrow{\ I_{X_1}\ }X_2\lra\cdots \lra X_{k-1}\xrightarrow{\ I_{X_{k-1}}\ } X_{k}
\end{equation}
be the tower of Iitaka fibrations as in diagram \eqref{eqn:diag1}, with morphisms $f_i\colon X_i\to A_i$ to quotient tori of $A$ and where $X_k$ is of general type or a point.
Then $X$ is an $f$-rational cohomology torus if, and only if, $X_k$ is an $f_k$-rational cohomology torus.
In particular, if $X_{k}$ is a point, $X$ is an $f$-rational cohomology torus.
\end{lemm}

\begin{defi}\label{def:ITT}
We say that $X$ is an \emph{Iitaka torus tower} if, in \eqref{eqn:ITT}, $X_k$ is a point.
\end{defi}

We give an example of an Iitaka torus tower which is not a torus.

\begin{exam}[An Iitaka torus tower which is not a torus]\label{exam:surface}
Let $\rho\colon  C\rightarrow E$ be a double cover of smooth projective curves, where $C$ has genus $g\geq 2$ and $E$ is an elliptic curve.
Let $\tau$ be the corresponding involution on $C$.
Let $E'\rightarrow E$ be a degree-$2$ \'etale cover of elliptic curves and let $\sigma$ be the corresponding involution on $E'$.
Let $X$ be the smooth surface $(C\times E')/\langle \tau\times\sigma\rangle$.
Then $X$ is an Iitaka torus tower but has Kodaira dimension 1, hence is not a torus.
\end{exam}

The answer to Catanese's original conjecture \cite[Conjecture 70]{cat} is therefore negative.
Nevertheless, we may still ask the following question.

\begin{qu}\label{que:qu}
Is a complex compact K\"{a}hler manifold
which is a rational cohomology torus always an Iitaka torus tower?
\end{qu}

To answer  (negatively)  this question,    we study   the variety $X_k$ of Lemma \ref{lem:char}, which is a possibly singular projective variety.

\subsection{Rational singularities}
Recall the following classical definition.

\begin{defi}
Let $X$ be a complex compact space and let $\mu\colon  X'\rightarrow X$ be a desingularization.
We say that $X$ has {\em rational singularities} if $R^i\mu_*\cO_{X'}=0$ for all $i>0$ and $\mu_*\cO_{X'}=\cO_X$ (equivalently, $X$ is normal).
\end{defi}

The following lemma explains the reason why we work with rational singularities.

\begin{lemm}\label{lem:rational}
Let $f\colon  X\tto A$ be a finite and surjective morphism from a projective variety $X$ with rational singularities to an abelian variety.
Consider a quotient $A\tto B$ of abelian varieties.
If the composition $X\tto A\tto B$ factors through a finite morphism $Y\rightarrow B$ with $Y$ normal, $Y$ has rational singularities.
\end{lemm}

In particular, the lemma applies when $X\to Y\to B$ is the Stein factorization of $X\to B$.

\begin{proof}
Since $f\colon  X\rightarrow A$ is finite, so is the induced morphism $g\colon X\rightarrow Y\times_BA$.
Since $Y\rightarrow B$ is finite and surjective, so is $ Y\times_BA\to A$, hence the image of $g$ is a component $\wY$ of $Y\times_BA$.
By \cite[Theorem~1]{kov-rat}, since $X$ has rational singularities, so has $\wY$.
Finally, since $A\to B$ is smooth, so is $\wY\hookrightarrow Y\times_BA\to Y$.
It follows that $Y$ has rational singularities.
\end{proof}

It follows that if $X$ has rational singularities, all the $X_i$ in the tower \eqref{eqn:ITT} have rational singularities.
Thus, in order to answer Question \ref{que:qu}, it suffices to answer the following.

\begin{qu}\label{que:qu1}
Can a projective variety with rational singularities which is a rational cohomology torus be of general type?
\end{qu}

To answer   (positively) this question, we first  prove that any desingularization must satisfy very strong numerical properties.

\begin{lemm}\label{lem:Gr0}
Let $X$ be a projective variety with
rational singularities.
For each $k$, we have an isomorphism
\[H^k(X, \cO_X)\isom \mathrm{Gr}_F^0H^k(X),\]
where $F^{\bullet}$ is the Hodge filtration for Deligne's mixed Hodge structure on $H^k(X)$.
\end{lemm}

\begin{proof}
By \cite[Theorem S]{kov-DB},
rational singularities are Du Bois.
If $\underline{\Omega}_X^\bullet$ is the Deligne-Du Bois complex of $X$ (\cite[Definition~7.34]{ps}), this means that $\underline{\Omega}_X^0$ is quasi-isomorphic to $\cO_X$.
By Deligne's theorem (\cite[Sections 8.1, 8.2, 9.3]{del3}, \cite[4.2.4]{ks}), the spectral sequence $E_1^{p,q}=\mathbb{H}^q(X, \underline{\Omega}_X^p)\Rightarrow H^{p+q}(X, \C)$ degenerates at $E_1$ and abuts to the Hodge filtration of Deligne's mixed Hodge structure.
Thus, we have $H^k(X, \cO_X)\isom H^k(X, \underline{\Omega}_X^0)=\mathrm{Gr}_F^0H^k(X)$.
\end{proof}

\begin{prop}\label{prop:sing-albanese}\label{cor:pg1}
Let $X$ be a projective variety with rational singularities.
Assume that $X$ is a rational cohomology torus.
Let $\mu\colon X'\rightarrow X$ be a desingularization.
\begin{itemize}
\item[1)] We have $h^k(X', \omega_{X'})=h^k(X, \omega_X)=\binom{n}{k}$; in particular, $\chi(X', \omega_{X'})=0$.
\item[2)] \label{prop:sing-albanese2} The Albanese morphism $a_{X'}\colon X'\rightarrow A_{X'}$ factors through $\mu$ and the induced morphism $X\to A_{X'}$ is the Albanese morphism of $X$ in the sense of Remark \ref{rmk:albanese}.
\end{itemize}
\end{prop}

\begin{proof}
By Proposition \ref{prop:albanese}, there is a finite morphism $f\colon X\rightarrow A$ to an \av\ such that $X$ is an $f$-rational cohomology torus.
Hence we have isomorphisms of Hodge structures $f^*\colon H^k(A)\isomlra H^k(X)$ for all $k$.
In particular, $\mathrm{Gr}_F^0H^k(X)\isom \mathrm{Gr}_F^0H^k(A)=H^k(A, \cO_A)$.
Since $X$ has rational singularities, we have $H^k(X, \cO_X)\isom \mathrm{Gr}_F^0H^k(X)$ (Lemma \ref{lem:Gr0}).
Hence $h^k(X, \cO_X)=\binom{n}{k}$.

Let $\mu\colon X'\rightarrow X$ be a desingularization.
Since $R\mu_*\omega_{X'}=\omega_X$, we have
\[h^k(X, \omega_X)=h^k(X', \omega_{X'})=h^{n-k}(X', \cO_{X'})=h^{n-k}(X, \cO_X)=\binom{n}{k}.\]

For 2), note that $h^1(X', \cO_{X'})=n$, hence $\dim (A_{X'})=\dim (X')=\dim (X)=n $.
By Proposition \ref{prop:albanese}, there is a quotient morphism $A_{X'}\rightarrow A$ with connected fibers, hence $A_X=A_{X'}$ and $a_{X'}$
factors through $\mu$.
\end{proof}

This simple but important proposition allows us to use the many known properties
of smooth varieties $X'$ of maximal Albanese dimension with $\chi(X', \omega_{X'})=0$ and $p_g(X')=1$.

\section{Rational cohomology tori in lower dimensions}

Thanks to the work of Chen--Debarre--Jiang \cite{CDJ} on smooth varieties of maximal Albanese dimension with $p_g=1$,
we can give a classification of rational cohomology tori up to dimension $3$.
We first recall some important examples.

\begin{exam}[Ein--Lazarsfeld \& Chen--Hacon threefolds]\label{exam:el&ch}
For each $j\in\{1,2,3\}$, consider an elliptic curve $E_j$ and a bielliptic curve $C_j \xrightarrow{2:1} E_j$ of genus $g_j \ge 2$, with corresponding involution
$\tau_j$ of $C_j$.
Set $A := E_1 \times E_2 \times E_3$ and consider the quotient $C_1 \times C_2 \times C_3\tto Z$ by the involution $\tau_1\times \tau_2 \times \tau_3$ and the tower of Galois covers
\[C_1 \times C_2 \times C_3 \stackrel{g}{\lraa}Z \stackrel{f}{\lraa}A\]
of degrees $2$ and $4$ respectively.
 The threefold $Z$ is  of general type with rational singularities, and it  has  $2^3\prod_{j=1}^3 (g_j-1)$ isolated singular points.
We call $Z$ an \emph{Ein--Lazarsfeld threefold} (\cite[Example 1.13]{el}).
If $ X\tto Z$ is any desingularization, we have $\chi(X,\omega_X)=\chi(Z,\omega_Z)=0$.

A variant of the previous construction gives us varieties with $p_g=1$, as follows.
Keeping the same notation, choose points $\xi_j\in\PE_j$ of order 2 and consider the induced double \'etale covers $E'_j\tto E_j$ and $C'_j\tto C_j$, with associated involution $\sigma_j$ of $C'_j$.
The involution $\tau_j$ on $C_j$ pulls back to an involution $\tau'_j$ on $C'_j$ (with quotient $E'_j$).
Let $C'_1\times C'_2\times C'_3\tto Z'$ be the quotient by the group (isomorphic to $(\Z/2\Z)^{4}$) of automorphisms generated by $\id_1\times \tau'_2\times\sigma_3$, $\sigma_1\times\id_2\times \tau'_3$, $\tau'_1\times\sigma_2\times\id_3$, and $\tau'_1\times\tau'_2\times\tau'_3$ and consider the tower
\[C'_1 \times C'_2 \times C'_3 \stackrel{g'}{\lraa}Z' \stackrel{f'}{\lraa}A\]
of Galois covers of respective degrees $2^4$ and $4$.
 The threefold  $Z'$  is  of general type and has rational singularities.
We call $Z'$ a \emph{Chen--Hacon threefold} (\cite[\S4, Example]{ch-Iitaka}).
For any desingularization $ X'\tto Z'$, one has $p_g(X')=1$.

The \'etale cover $E'_1\times E'_2\times E'_3\tto E_1\times E_2\times E_3$ pulls back to an \'etale cover $Z''\to Z'$, where $Z''$ is an Ein--Lazarsfeld threefold; in particular, $Z'$ also has isolated singular points.
Moreover, the quotient of $C'_1\times C'_2\times C'_3$ by the group of automorphisms generated by $\id_1\times \tau'_2\times\sigma_3$, $\sigma_1\times\id_2\times \tau'_3$, $\tau'_1\times\sigma_2\times\id_3$ is a smooth double cover of $Z'$, since the group acts freely.

This terminology differs from that of \cite{CDJ}: there, Ein--Lazarsfeld and Chen--Hacon threefolds refer to any of their desingularizations.
Our singular threefolds can be obtained from their smooth versions by considering the Stein factorizations of their Albanese morphisms.
\end{exam}

We can now prove the classification of rational cohomology tori up to dimension $3$ stated in the introduction.

\begin{proof}[{Proof of Theorem \ref{thm:dim2-3}}]
Let $X$ be a compact K\"ahler space with rational singularities which is a rational cohomology torus.
By Lemma \ref{lem:char}, we may assume that $X$ is projective.
Let
$\mu\colon X'\rightarrow X$ be a desingularization.
By Proposition \ref{cor:pg1}, we have $\chi(X', \omega_{X'})=0$ and $p_g(X')=1$.

If $\dim (X)=2$ and $X$ is of general type, we have $\chi(X', \omega_{X'})>0$ by Riemann-Roch, which is a contradiction.
If $\kappa(X)=0$, $X$ is an abelian variety by \cite[Corollary 2]{kaw}.
If $\kappa(X)=1$, in the diagram \eqref{eqn:diag1} of Remark \ref{rmk:algebraic}, $Y$ is an elliptic curve.
Hence $X$ is an Iitaka torus tower.

Assume $\dim (X)=3$.
If $X'$ is of general type, we can apply the structure theorem \cite[Theorem 6.3]{CDJ}: there exists an abelian \'etale cover $\wA\tto A_{X'}$ such that, in the Stein factorization $X'\times_{A_{X'}}\wA\to \wX\to \wA$, the variety $\wX$
is a Chen--Hacon threefold (Example \ref{exam:el&ch}).
As noted in Proposition \ref{prop:sing-albanese}.2),
$X$ appears in the Stein factorization of the Albanese morphism of $X'$, hence $\wX$ is an \'etale cover of $X$, and $X$ is singular.
We then apply Lemma \ref{lem:char} and part 1) to get the first part of 2).

For the second part of 2), it suffices to show that a Chen--Hacon threefold is a rational cohomology torus.
This follows from the more general Proposition \ref{prop:construction1} below.

For 3), we note that there exists a smooth projective threefold $\hat{Y}$ with an involution $\tau$ such that $Y:=\hat{Y}/\langle\tau\rangle$ is a Chen--Hacon threefold (Example \ref{exam:el&ch}), hence   a rational cohomology torus.
Let $\sigma$ be a translation of order $2$ on any non-zero abelian variety $K$.
The involution $\tau\times \sigma$ acts freely on $\hat{Y}\times K$ and $X:=(\hat{Y}\times K)/\langle \tau\times \sigma\rangle$ is a smooth projective variety.
Moreover, the natural morphism $X\rightarrow Y $ is the Iitaka fibration in \eqref{eqn:diag1}.
Thus $X$ is also a rational cohomology torus.
Since $Y$ is of general type, $X$ is not an Iitaka torus tower.

Moreover, by Example \ref{exam-dim}, there exists a rational cohomology torus $Y$ with rational singularities and  of general type in any dimension $\geq 3$ with a smooth double  cover  $\hat{Y}$.
By the same construction, whenever $3\leq m\leq n-1$, there exists a smooth rational cohomology torus $X$ of dimension $n$ with $\kappa(X)=m$.
\end{proof}

The following proposition shows in particular that Chen--Hacon threefolds are rational cohomology tori.

\begin{prop} \label{prop:construction1}
For each $j\in\set{1,\dots,n}$, let $\rho_j\colon C_j\rightarrow E_j$ be an abelian Galois cover with group $G_j$, where $C_j$ is a smooth projective curve and $E_j$ an elliptic curve.
Take a subgroup $G$ of $G_1\times\cdots\times G_n$ and set
$X:=(C_1\times\cdots \times C_n)/G$.
Assume $h^0(X, \omega_X)=1$; then $X$ is a rational cohomology torus with rational singularities.
\end{prop}

\begin{proof}
Set $V:=C_1\times \cdots \times C_n$ and $A:=E_1\times\cdots \times E_n$, and let
\[ \rho\colon V\stackrel{f}{\lthra} X\stackrel{g}{\lthra} A \]
be the quotient morphisms.
The variety $X$ has finite quotient singularities, which are rational singularities.
In particular, $H^k(X)$ has a pure Hodge structure for all $k\in\set{0,\dots,n}$ (\cite[Theorem 2.43]{ps}).
More precisely, if $\iota\colon X_{\mathrm{reg}}\hookrightarrow X$ is the smooth locus of $X$ and we set
$\widetilde{\Omega}_{X}^p:=\iota_*(\Omega_{X_{\mathrm{reg}}}^p)$,
we have $\widetilde{\Omega}_{X}^p=(f_*\Omega_V^p)^G$ and $\widetilde{\Omega}_{X}^\bullet$ is a resolution of the constant sheaf $\C_X$ (\cite[Lemma 2.46]{ps}).
Thus, we have $H^q(X,\widetilde{\Omega}_{X}^p)=\Gr^q_{F}H^{p+q}(X, \C)$ (\cite[Proof of Theorem 2.43]{ps}).

We may assume that each projection $G\rightarrow G_j$ is surjective.
Indeed, if we denote by $H_j$ the image of this projection, there are natural morphisms $X\rightarrow C_j/H_j\to E_j$.
Since $X$ has maximal Albanese dimension, the condition $h^0(X, \omega_X)=1$ implies $h^0(C_j/H_j, \omega_{C_j/H_j})=1$ (\cite[Lemma 2.3]{JLT1}), so $C_j/H_j$ is also an elliptic curve.
Then we simply replace $G_j$ by $H_j$ and $E_j$ by $C_j/H_j$.

Let $j\in\set{1,\dots,n}$.
Since $\rho_j$ is an abelian Galois covering, we may write
\begin{eqnarray*}
&& \rho_{j*}\omega_{C_j}=\cO_{E_j}\oplus \bigoplus_{1\neq \chi_j\in G_j^\vee}\cL_{\chi_j},\\\
&&\rho_{j*}\cO_{C_j}=\cO_{E_j}\oplus \bigoplus_{1\neq \chi_j\in G_j^\vee}\cL_{\chi_j}^{-1},
\end{eqnarray*}
where $G_j^\vee$ is the character group of $G_j$ and $\cL_{\chi_j}$ is the line bundle  on $E_j$ associated with the character $\chi_j\in G_j^\vee$.
Since $G\rightarrow G_j$ is surjective, its dual $G_j^\vee\to G^\vee$ is injective.

Since $\rho_{j*}\omega_{C_j}$ is a nef vector bundle on $E_j$, each line bundle $\cL_{\chi_j}$ is nef, and for
$ \chi_j\ne 1$, it is either ample or non-trivial torsion in $\PE_j$.
Moreover,
if $\cL_{\chi_j}$ is a torsion line bundle, so is $\cL_{\chi_j^m}=\cL_{\chi_j}^{\otimes m}$ for each $m\in \Z$.
Thus, $\cL_{\chi_j}$ is non-trivial torsion if, and only if, $\cL_{\chi_j^{-1}}=\cL_{\chi_j}^{-1}$ is non-trivial torsion.

We compute
\begin{eqnarray*}
g_*\omega_{X}&=&g_*\big((f_*\omega_V)^G\big)=(\rho_*\omega_V)^G\\
&=& \Big(\hskip-4mm\bigoplus_{\chi_j\in G_j^\vee,\ 1\leq j\leq n} (\cL_{\chi_1}\boxtimes\cdots \boxtimes \cL_{\chi_n})\Big)^G\\
&=& \bigoplus_{ \chi_j\in G_j^\vee,\ \chi_1\dots\chi_n=1\in G^\vee } (\cL_{\chi_1}\boxtimes\cdots \boxtimes \cL_{\chi_n}).
\end{eqnarray*}

Since $\cO_A$ is a direct summand of $g_*\omega_X$ and $h^0(X, \omega_X)=1$, we conclude that for any $ (\chi_1,\ldots, \chi_n)\in G_1^\vee\times\cdots \times G_n^\vee$ not all trivial such that $\chi_1\cdots\chi_n=1\in G^\vee$, at least one of the corresponding line bundles $\cL_{\chi_j}$ is non-trivial torsion.

For any subset $J=\{j_1, \ldots, j_p\}$ of $T:=\{1, \ldots, n\}$, we set $V_J:=C_{j_1}\times \cdots \times C_{j_p}$ and we let $p_J\colon V\rightarrow V_J$ be the projection.
We also denote by $q_j\colon A\rightarrow E_j$ the projections.
Then,
\begin{eqnarray*}
g_*\widetilde{\Omega}_X^p &=& g_*\big((f_*\Omega_V^p)^G\big)=(\rho_*\Omega_V^p)^G\\
&=& \big(\rho_*\hskip-2mm\bigoplus_{J\subset T,\ |J|=p}p_J^*\omega_{V_J}\big)^G\\
&=&
\Big(\bigoplus_{J\subset T,\ |J|=p}
\Big(\bigoplus_{\chi_j\in G_j^\vee{\rm\ for\ all\ }j\in J}\big(\bigotimes_{j\in J} q_j^*\cL_{\chi_j}\big)\Big)\otimes
\Big(\bigoplus_{\chi_k\in G_k^\vee{\rm\ for\ all\ }k\in J^c}\big(\bigotimes_{k\in J^c} q_k^*\cL_{\chi_j}^{-1}\big)\Big)^G\\
&=& \bigoplus_{\substack{J\subset T\\|J|=p}}\ \ \bigoplus_{\substack{\chi_j\in G_j^\vee,\ \chi_k\in G_k^\vee\\\prod_{j\in J}\chi_j\prod_{k\in J^c}\chi_k^{-1}=1\in G^\vee}}\big(\bigotimes_{\substack{\chi_j\in G_j^\vee\\j\in J}}q_j^*\cL_{\chi_j}\big)\otimes \big(\bigotimes_{\substack{\chi_k\in G_k^\vee\\k\in J^c}}q_k^*\cL_{\chi_k}^{-1}\big).
\end{eqnarray*}

For example, for $J=\set{1,\ldots,p}$, the fourth equality reads
\[\rho_*p_J^*\omega_{V_J}
=\rho_{1*}\omega_{C_1}\boxtimes\cdots\boxtimes \rho_{p*}\omega_{C_p}\boxtimes\rho_{p+1*}\cO_{C_{p+1}}\boxtimes\cdots\boxtimes \rho_{n*}\cO_{C_{n}}.\]

For any non-trivial solution of $\prod_{j\in J}\chi_j\prod_{k\in J^c}\chi_k^{-1}=1\in G^\vee$, we have already seen that the condition $h^0(X,\omega_X)=1$ implies that either there exists $j\in J$ such that $\cL_{\chi_j}$ is non-trivial torsion or there exists $k\in J^c$ such that $\cL_{\chi_k^{-1}}$ is non-trivial torsion, in which case $\cL_{\chi_k}^{-1}$ is also non-trivial torsion.
Thus, by the K\"unneth formula, only the trivial direct summands of $g_*\Omega_X^p$ have non-trivial cohomology groups, since all the others contain a non-trivial torsion line bundle.
Therefore, \[H^q(X, \widetilde{\Omega}_X^p)=H^q(A, g_*\widetilde{\Omega}_X^p)=H^q(A, \Omega_A^p).\]
Since $H^q(X,\widetilde{\Omega}_{X}^p)=\mathrm{Gr}^q_{F}H^{p+q}(X, \C)$, we conclude that $X$ is a rational cohomology torus.
\end{proof}

\section{Constraints on the Albanese morphism}

In this section, we study how the condition $\chi(X,\omega_X)=0$ gives restrictions on the degree of the Albanese morphism.
Note that $\chi(X,\omega_X)$ is a birational invariant.
It would be interesting to study the non-birational conditions $\chi(X,\Omega_X^p)=0$ for $0< p<\dim(X)$ to get further restrictions on the structure of rational cohomology tori.

We first recall some facts about varieties of general type, of maximal Albanese dimension, and with $\chi=0$ (for more details, see \cite{CDJ, CJ}).

Let $X$ be a smooth projective variety and let $f\colon X\rightarrow A$ be a generically finite morphism to an abelian variety.
We set
\[V^i(f_*\omega_X):=\{[P]\in \PA\mid H^i(A, f_*\omega_X\otimes P)\neq 0\}. \]
By \cite{gl1,gl2, sim, hac}, $V^i(f_*\omega_X)$ is a union of torsion translates of abelian subvarieties of $\PA$   of codimension $\ge i$.
The set
\begin{equation}\label{eqn:S_f}
S_f:=\set{  T\subset \PA\left|\begin{array}{c} \exists\,i\ge1
\text{ such that }   T \text{ is a component}\\\text{of }V^i(f_*\omega_X)\text{ with }\codim_\PA (  T)=i\end{array}\right.}
\end{equation}
controls the positivity of the sheaf $f_*\omega_X$ (\cite{CJ}, \cite[Section 3]{JLT1}).  
\medskip

We use the following notation: for any abelian subvariety $\PB\subset \PA$, we let
\begin{equation}\label{eqn:X_B}
X\lthra X_B\xrightarrow{\ f_B\ } B
\end{equation}
be the Stein factorization of the composition $X \xrightarrow{\ f\ } A\tto B$.
After birational modifications, we may assume that $X_B$ is also smooth.
Note that when $f$ is the Albanese morphism of $X$ and $\PB\in S_f$, the map $f_B$ is   the Albanese morphism of $X_B$.

\begin{lemm}[{\cite{gl1,el,CDJ,CJ}}]\label{lem:chi0}
Let $X$ be a smooth projective variety of general type with a generically finite morphism $f\colon X\rightarrow A$ to an abelian variety.
\begin{itemize}
\item[1)] We have $\chi(X, \omega_X)=0$ if, and only if, $V^0(f_*\omega_X)$ is a proper subset of $\PA$. If these properties hold, the \av\  $A $ has at least $3$ simple factors.  
\item[2)]  If $\chi(X, \omega_X)=0$ and  $T$ is an irreducible component of $V^0(f_*\omega_X)$, we have $T\in S_f$.
More precisely,   $T$ is an irreducible component of $V^i(f_*\omega_X)$, where $i=\codim_{\PA}(T)$.  
\item[3)] For any abelian variety $\PB\in S_f$, the variety $X_B$ is of general type and $\chi(X_B, \omega_{X_B})>0$.
\item[4)] If $\chi(X, \omega_X)=0$, there exists a quotient $A\tto B$ of abelian varieties such that the general fiber $F_B$ of the induced morphism
$X\tto X_B$ is primitive with $\chi(F_B, \omega_{F_B})=0$.
\end{itemize}
\end{lemm}

\begin{proof}
 The equivalence in  1) follows   from generic vanishing (\cite[Theorem 1]{gl1}, \cite[Remark 1.6, Theorem 1.2]{el})  and the other statement is \cite[Corollary 3.4]{CDJ}. 
For 2), see \cite[Claim (1.10)]{el}.
For 3), see \cite[Theorem 3.1]{CDJ}.
The last statement is \cite[Proposition 6.2]{CJ}.
\end{proof}

We introduce the notion of {\em primitive} and {\em minimal primitive} varieties.
In Section \ref{ssec:str_minimal_prim}, we study the structure of minimal primitive varieties and prove Theorem \ref{thm:deg-chi0}.
We provide examples in Section \ref{sec:examples}.

\begin{defi}[{\cite[Definition 6.1]{CJ}}]\label{def:primitive}
Let $X$ be a smooth projective variety of general type, of maximal Albanese dimension, and with $\chi(X,\omega_X)=0$.
We say that $X$ is \emph{primitive} if there exist no proper smooth subvarieties $F$ through a general point of $X$ such that $\chi(F,\omega_F)=0$.

We say that $X$ is \emph{minimal primitive} if it is primitive and, for any rational factorization $ X\stackrel{a}{\dashrightarrow} Y\rightarrow A_X$ of the Albanese morphism of $X$
through a smooth projective variety $Y$ of general type, the map $a$ is birational.
\end{defi}

This definition is different from the definition of primitive given in \cite[Definition 1.24]{cat-moduli}, but is equivalent to the one given in \cite[Definition 6.1]{CJ}.

We will use the following results concerning primitive varieties.

\begin{lemm}[{\cite{CDJ,CJ}}]\label{lem:prim}
Let $X$ be a smooth projective variety of general type, of maximal Albanese dimension, with $\chi(X,\omega_X)=0$.
Assume that $X$ is primitive.
\begin{itemize}
\item[1)] The Albanese morphism $a_X\colon X\to A_X$ is surjective.
\item[2)] For any quotient $A_X\tto B$  to a simple abelian variety, with connected fibers, the composition $X \xrightarrow{a_X} A_X\tto B$ is a fibration.
\item[3)] If the abelian variety $A$ has $m$ simple factors, $V^0(f_*\omega_X)$ has at least $m$ irreducible components; each component is a torsion translate of an abelian variety with $m-1$ simple factors, and the intersection of these components has dimension $0$.\label{lem(4):prim}
\end{itemize}

\end{lemm}

\begin{proof}
Assertion 1) follows   from \cite[Lemma 4.6]{CDJ} and the definition of  primitive. 
For 2), see \cite[Lemma 6.4]{CJ}.
Statement 3) also follows from \cite[Lemma 6.4]{CJ}, since for any quotient $\PA_X\tto \PK$ of abelian varieties, the composition $V^0(f_*\omega_X)\hookrightarrow A_X\tto \PK$ is surjective.
\end{proof}

\subsection{The structure of minimal primitive varieties}\label{ssec:str_minimal_prim}
We describe the structure of minimal primitive varieties $X$:
we prove that each simple factor $K_j$ of the Albanese variety $A_X$ has a birational Galois cover $F_j\to K_j$ with Galois group $G_j$ such that
$X$ is a quotient of the product of the $F_j$ by a finite subgroup of the product of the $G_j$.
When the $F_j$ are curves, these quotient varieties already played an important role in Proposition \ref{prop:construction1}.

\begin{theo}\label{thm:minimal}
Let $X$ be a smooth projective variety of general type, of maximal Albanese dimension, with $\chi(X,\omega_X)=0$.
We assume that $X$ is minimal primitive.
For some $m\ge3$, there exist
\begin{itemize}
\item smooth projective varieties $F_1,\dots,F_m$ of general type,
\item non-trivial finite groups $G_j$ acting faithfully on $F_j$ such that the quotient $F_j/G_j$ is birational to a simple abelian variety $K_j$,
\item an isogeny $K_1\times\cdots\times K_m\rightarrow A_X$ which induces an \'etale cover $\widetilde{X}\to X$,
\item a subgroup $G$ of $G_1\times\cdots\times G_m$,
\end{itemize}
such that $\widetilde{X}$ is birational to $(F_1\times\cdots\times F_m)/G$.

Furthermore, we can assume that the projections\footnote{Here, the notation $\widehat{G}_i$ means that the factor $G_i$ is missing in the product.} $\pi_{ij}\colon G\rightarrow G_1\times\cdots\times \widehat{G}_i\times\cdots\times \widehat{G}_j\times \cdots \times G_m$ are injective and the projections $G\tto G_i$ are surjective whenever $1\leq i<j\leq m$.
\end{theo}

We summarize part of the conclusions of the theorem in a commutative diagram:
\begin{equation*}\label{ndiag}
\xymatrix
@C=12mm@R=3mm{
F_1\times\dots\times F_m\ar@{->>}[rrdd]^-{\ \ \ /G_1\times\dots\times G_m}\ar@{-->>}[dd]^-{/G}\\\\
\widetilde X\ar@{->>}[rr]\ar@{->>}[dd]^{\rm\acute etale}&&K_1\times\dots\times K_m\ar@{->>}[dd]^{\rm\acute etale}\\
&\square\\
X\ar@{->>}[rr]^{a_X} && A_X.}
\end{equation*}

Minimal primitive varieties with $\chi=0$ will be constructed in Examples \ref{exam:chi0-p2} and \ref{exam:rt-p2}. 

\begin{proof}
The proof is divided into four steps. 
\phantom\qedhere
\end{proof}
\subsubsection*{\bf Step 1} {\em Reduction via \'etale covers.}\\
Taking if necessary an \'etale cover of $A_X$ (which induces an \'etale cover of $X$), we may assume that each element of $S_{a_X}$ (see \eqref{eqn:S_f}), which by Lemma \ref{lem:chi0}.1) and 2) is nonempty, contains the origin $0_{\PA_X}$.

Assume that $A_X$ has $m$ simple factors.
Since we are assuming that $X$ is {\em primitive}, by Lemma \ref{lem:prim}.3), there exist irreducible components $\PA_1, \ldots, \PA_m$ of $V^0(a_{X*}\omega_X)$ such that each $\PA_j$ has $m-1$ simple factors and
\begin{equation}\label{eqn:int=0}
\dim \Big(\bigcap_{1\leq j\leq m}\PA_j\Big)=0.
\end{equation}
The quotient $\PK_j:=\PA_X/\PA_j$ is a simple abelian variety and we have the dual injective morphism $K_j\hookrightarrow A_X$.
By \eqref{eqn:int=0}, the sum morphism
\[\pi\colon  A':=K_1\times\cdots\times K_m\rightarrow A_X\]
is an isogeny.

If $X'\tto X$ is the \'etale cover induced by $\pi$, we have $\pi^*(\PA_j)=\PK_1\times\cdots \times\{0_{\PK_j}\}\times\cdots \times\PK_m$.
Thus $V^0(a_{X'*}\omega_{X'})$ contains at least the $m$ components $ \PK_1\times\cdots \times\{0_{\PK_j}\}\times\cdots\times \PK_m$.
Moreover, $A_{X'}=K_1\times\cdots\times K_m$.

\smallskip

Thus, we have constructed the following elements in the statement of the theorem:
the simple abelian varieties $K_i$ and the isogeny $K_1\times\cdots\times K_m\rightarrow A_X$.
We still need to identify the fibers $F_j$ and the groups $G_j$ and $G$.

\subsubsection*{\bf Step 2} {\em A special property of fiber products.}\\
By Step 1, we can suppose that $\PA_j:=\PK_1\times\cdots \times\{0_{\PK_j}\}\times\cdots\times \PK_m$ is a component of $V^0(a_{X*}\omega_{X})$ and $A_{X}=K_1\times\cdots\times K_m$.

For each $1\leq i< j\leq m$, set $\PA_{ij}:=\PA_i\cap\PA_j$.
Using the notation \eqref{eqn:X_B}, we have a commutative diagram
\begin{equation}\label{eqn:morphisms}
\xymatrix
@!=.8cm
@C=20mm@R=5mm{
& X\ar@{->>}[ddl]_{f_i}\ar@{->>}[ddr]^{f_j}\ar@{->>}[d]^-{f_{ij}}\\
&Y_{ij}\ar@{->>}[ld]\ar@{->>}[rd]\\
X_{A_i}\ar@{->>}[dr]^{g_{ij}} && X_{A_j}\ar@{->>}[dl]_{g_{ji}}\\
& X_{A_{ij}},}
\end{equation}
where $Y_{ij}$ is a desingularization of the main component of $X_{A_i}\times_{X_{A_{ij}}}X_{A_j}$.
Since $A_{X}=K_1\times\cdots\times K_m$, we have $A_i\times_{A_{ij}}A_j\isom A_X$.
Hence, the Albanese morphism factors as
\begin{equation}
a_X\colon  X\xrightarrow{\ f_{ij}\ } {Y}_{ij}\rightarrow A_X.
\end{equation}
Since $X_{A_i}$ and $X_{A_j}$ are of general type by Lemma \ref{lem:chi0}.3), $Y_{ij}$ is of general type by Viehweg's sub-additivity theorem (\cite[Corollary IV]{vie}).
Moreover, the assumption that $X$ is \emph{minimal} implies that $f_{ij}$ is birational.

In other words, $X$ is birational to the fiber product of any two of the fibrations induced by the simple factors of the Albanese variety.

\subsubsection*{\bf Step 3} {\em The $f_j$ are all birationally isotrivial fibrations.}\\
We use the notation of diagram \eqref{eqn:morphisms}.
Since $f_{12}$ is birational, for a general point $x\in X_{A_1}$, the fiber $F_x$ of $f_1$ is birational to $g_{21}^{-1}(g_{12}(x))$.
Hence $F_x$ is birational to $F_y$ for $y\in g_{12}^{-1}(g_{12}(x))$ general.
Similarly, $F_x$ is birational to $F_y$ for $y\in g_{1j}^{-1}(g_{1j}(x))$ general, for all $j\in\set{2,\dots,m}$.
Any two points of $X_{A_1}$ can be connected by a chain of fibers of $g_{12}$, $g_{13},\dots,$ or $g_{1m}$.
For general points $x$ and $y$ of $X_{A_1}$, $F_x$ is therefore birational to $F_y$ and $f_1$ is a birationally isotrivial fibration.

By the same argument, we see that $f_j$ is a birationally isotrivial fibration for each $j\in\set{1,\dots,m}$.
We denote by $F_j$ its general fiber; since $X$ is of general type, so is $F_j$.
\smallskip

We have now constructed the varieties $F_j$ in the statement of the theorem.
It remains to see that there are finite groups $G_j$ acting faithfully on $F_j$ such that $F_j/G_j$ is birational to $K_j$ and $X$ is birational to the quotient $(F_1\times\cdots\times F_m)/G$ for some subgroup $G$ of $G_1\times\cdots \times G_m$.

\subsubsection*{\bf Step 4}
{\em The finite groups $G_1,\dots , G_m$  and  the subgroup $G$ of  $G_1\times\cdots \times G_m$.}\\
The following lemma allows us to characterize varieties which are finite group quotients of a product of varieties and finishes the proof of Theorem \ref{thm:minimal}.

\begin{lemm}\label{lem:str}
Let $f\colon X\to V_1\times \cdots \times V_m$ be a generically finite and surjective morphism between normal projective varieties.
Assume that $X$ is of general type and that, for each $j\in\set{1,\dots,m}$, there is a commutative diagram\footnote{In this lemma, $\widehat{V_j}$ means that the factor $V_j$ is omitted in the cartesian product.}
\[
\xymatrix{
X\ar@{->>}@/^1.5pc/[rr]^{g_j}\ar@{->>}[r]^(.45)f\ar@{->>}[d]_{f_j}\ar@{->>}[rd]_(.35)\varphi & V_1\times \cdots \times V_m\ar@{->>}[d]\ar@{->>}[r]& V_j\\
X_j \ar@{->>}[r] &V_1\times \cdots\times\widehat{V}_j\times\cdots \times V_m,
}
\]
where $X\stackrel{f_j}{\tto} X_j$ is the Stein factorization of $\varphi$, the variety $X_j$ is of general type, $f_j$ is a birationally isotrivial fibration with general fiber $F_j$, and $g_j$ is a fibration.

Then, there exists a finite group $G_j$ acting faithfully on $F_j$ such that $F_j/G_j$ is birational to $V_j$.
Moreover, $X$ is birational to a quotient $(F_1\times\cdots\times F_m)/G$, where $G$ is a subgroup of $G_1\times\cdots \times G_m$ with surjective projections $G\tto G_j$.
\end{lemm}

Before giving the proof of the lemma, note that it applies to our situation 
\[
	f\colon X\to K_1\times\cdots\times K_m,
\]
thanks to Step 3 and Lemma \ref{lem:prim}.2), which  ensures that the $g_j\colon X\to K_j$ are fibrations.

\begin{proof}
Let $\wX_1$ be a general fiber of $g_1 \colon  X\tto V_1$ and let $\res{f}{\wX_1}\colon  \wX_1\tto V_2\times\cdots\times V_m$ be the induced generically finite morphism.
For $j\in\set{2,\dots,m}$, denote by $\wX_1\tto V_j'\tto V_j$ the Stein factorization of the natural morphism $\wX_1\tto V_j$.
The induced morphism $f'\colon  \wX_1\to V'_2\times\cdots\times V'_m$ is generically finite and surjective.
For each $j\in\set{2,\dots,m}$, let $\wX_1 \stackrel{f'_j}{\tto} Y_j\tto V'_2\times \cdots\times\widehat{V'_j}\times\cdots \times V'_m$ be the Stein factorization of the natural morphism.
We summarize these constructions in the commutative diagram
\[
\xymatrix
@M=4pt{
\wX_1\ar@{_{(}->}[d]\ar@{->>}[r]^{f'_j}& Y_j\ar@{->>}[r]\ar[d] & V'_2\times \cdots\times\widehat{V'_j}\times\cdots \times V'_m\ar@{->>}[r]\ar[d] & \set{*}\ar@{_{(}->}[d]\\
X\ar@{->>}@/_1.5pc/[rrr]_{g_1}\ar@{->>}[r]^{f_j}&X_j\ar@{->>}[r] &V_1\times \cdots\times\widehat{V}_j\times\cdots \times V_m\ar@{->>}[r]& V_1,
}
\]
where the second and third vertical arrows are finite.
Since the images of the $Y_j$ in $X_j$ cover $X_j$ and $X_j$ is of general type by hypothesis, $Y_j$ is also of general type; similarly, $\wX_1$ is also of general type.
Moreover, $f_j'$ is also a birationally isotrivial fibration with general fiber $F_j$.

The morphism $f'\colon \wX_1\rightarrow V'_2\times\cdots\times V'_m$ satisfies again the hypotheses of the lemma.
Thus, by induction on $m$, we obtain that $\wX_1$ is birational to a quotient of $F_2\times\cdots\times F_m$.
Since a fixed variety can only dominate finitely many birational classes of varieties of general type, $g_1\colon X\to V_1$ is birationally isotrivial.

Thus, after a suitable finite Galois base change $F'_1\rightarrow V_1$ with Galois group $G_1$, where $F_1'$ is normal, we have a birational isomorphism $F'_1\times \wX_1\isomdra F'_1\times_{V_1}X$.
Since $\wX_1$ is of general type, its birational automorphism group
is finite.
The action of $G_1\times \id$ on $F'_1\times_{V_1}X$ therefore induces a birational action of $G_1 $ on $\wX_1$ such that $X$ is birationally the quotient of $F'_1\times \wX_1$ by the diagonal action of $G_1$.

Note that $G_1$ acts on the canonical models of $\wX_1$ and $F'_1$.
After an equivariant resolution of singularities \cite[Theorem 0.1]{aw}, we may assume that $\wX_1$ and $F'_1$ are smooth, still with $G_1$-actions, and that $G_1$ acts faithfully on $F'_1$.
We have the commutative diagram
\begin{equation*}
\xymatrix@M=4pt{
F'_1\times \wX_1\ar@{->>}[r] \ar@/_1.7pc/[rr]_-{\pi} &(F'_1\times \wX_1)/G_1\ar@{-->}[r]^-\sim & X \ar@{->>}[r]^-f \ar@{->>}[d]^-{f_1}& V_1\times\cdots\times V_m\ar@{->>}[d]\\
&& X_1\ar@{->>}[r] &V_2\times\cdots\times V_m.}
\end{equation*}
Let $x\in F'_1$ be a general point; there is a dominant rational map $f_{1x}=f_1\circ\res{\pi}{\set{x}\times \wX_1}\colon  \wX_1\dashrightarrow X_1$.
Since any family of dominant maps between varieties of general type is locally constant, we have $f_{1x}= f_{1y}$ for $x$ and $y$ general points of $F'_1$.
Thus, $f_1$ contracts the image of $F_1'$ in $X$ and $F_1'\isomdra F_1$, hence $F_1\rightarrow V_1$ is a birational Galois cover with Galois group $G_1$.

Similarly, for each $j\in\set{1,\dots,m}$, the map $F_j\rightarrow V_j$ is a birational Galois cover with Galois group $G_j$ and there are dominant maps
\begin{equation*}
\xymatrix{
F_1\times\cdots\times F_m\ar@{-->}[r]\ar@/^2pc/[rr]^{G_1\times\cdots\times G_{m}}_{\mathrm{Galois}} & X \ar[r]& V_1\times\cdots \times V_m.
}
\end{equation*}
Thus there is a subgroup $G\subset G_1\times\cdots\times G_m$ such that $X$ is birational to $(F_1\times\cdots\times F_m)/G$.
Since $g_j\colon X\rightarrow V_j$ is a fibration, the projection $G\rightarrow G_j$ is surjective for each $j\in\set{1,\dots,m}$.
\end{proof}

To finish the proof of Theorem \ref{thm:minimal}, it only remains to prove the injectivity assertion of the projection to the product with two factors missing.
This follows from the minimality assumption: the morphisms $f_{ij}\colon  X\rightarrow Y_{ij}$ are birational, thus a general fiber of $X\rightarrow X_{A_{ij}}$ is birational to $F_i\times F_j$.
This implies that the projections $ \pi_{ij} $ are injective.\qed

\subsection{Divisibility  properties  of the degree of the Albanese morphism}\label{ssec:deg}
The main result of this section is the following theorem.

\begin{theo}\label{thm:deg-chi0}
Let $X$ be a smooth projective variety of general type and maximal Albanese dimension and let $a_X$ be its Albanese morphism.
If $\chi(X, \omega_X)=0$, there exists a prime number $p$ such that $p^2$ divides the degree of $a_X$ onto its image.
\end{theo}

By Proposition \ref{prop:sing-albanese}, if $X$ is a rational cohomology torus of general type with rational singularities, we have $\chi(X, \omega_X)=\chi(X', \omega_{X'})=0$ for any desingularization $ X'\rightarrow X$.
Thus Theorem \ref{thm:deg-rct} in the introduction directly follows Theorem \ref{thm:deg-chi0}.

\begin{proof}
We first reduce to the case of minimal primitive varieties (see Definition \ref{def:primitive}) using induction on the dimension.
Then,  we apply   Theorem \ref{thm:minimal} and study the numerical properties  of the degree of the Albanese morphism.\phantom\qedhere
\end{proof}

\subsubsection*{\bf Step 1} {\em Reduction to minimal primitive varieties.}\\
We may assume that $X$ is primitive with $\chi(X, \omega_X)=0$.
Otherwise, by Lemma \ref{lem:chi0}.4) (see also \eqref{eqn:X_B}), there exists a quotient $A_X\tto B:=A_X/K$ such that the general fiber $F$ of the induced fibration $X\tto X_B$ is primitive with $\chi(F ,\omega_{F})=0$.
The restriction $\res{a_X}{F}\colon  F\rightarrow K$ factors through the (surjective) Albanese morphism of $F$ and we can argue by induction on the dimension.

Moreover, if $a_{X} $ factors as $X\dashrightarrow Y\to A_X$ through a variety of general type $Y$, after birational modifications, we may assume that we have morphisms between smooth projective varieties
\[a_{X}\colon  X\tto Y\xrightarrow{\ a_Y\ } A_{X}=A_Y.\]
Therefore, we can replace $X$ by $Y$ and study the structure of $a_Y$.
Note that $Y$ may not be {primitive} and we need to reapply induction on the dimension as before.
Finally, we get an $X$ which is a minimal primitive variety of general type.

\medskip

The structure of $a_X$ remains the same after taking abelian \'etale covers of $X$.
Thus, using Theorem \ref{thm:minimal}, we can assume that $X$ is birational to
\[(F_1\times\cdots\times F_m)/G,\]
where $F_j$ is a smooth projective variety acted on faithfully by the finite group $G_j$ such that $F_j/G_j$ is birational to a simple abelian variety $K_j$,
$G$ is a subgroup of $G_1\times\cdots\times G_m$,
and
\[A_X= K_1\times\cdots\times K_m.\]
Furthermore, we can assume that for all $1\leq i<j\leq m$, the projection $G\rightarrow G_1\times\cdots\times \widehat{G}_j\times\cdots\times \widehat{G}_j\times \cdots \times G_m$ is injective and the projection $G\tto G_j$ is surjective.

\subsubsection*{\bf Step 2} {\em Computation of $\deg (a_X)$.}\\
Set $g:=|G|$.
Let $j\in\set{1,\dots,m}$ and set
$g_j:=|G_j|$.
Since $A_X=K_1\times\cdots\times K_m$ and $K_j$ is birational to $F_j/G_j$, we have
\[\deg (a_X) =\frac{1}{g}\prod_{1\leq j\leq m}g_j.\]
Moreover, since the projection $G\tto G_j$ is surjective, we have $g_j\mid g$, hence
\begin{equation}\label{formp}
\deg (a_X)\,\Big|\prod_{2\leq j\leq m}g_j.
\end{equation}
Finally, since the projections $G\rightarrow G_1\times\cdots\times \widehat{G}_j\times\cdots\times \widehat{G}_k\times \cdots \times G_m$ are injective, we have
\begin{equation}\label{formp1}
g_jg_k\mid \deg (a_X) \qquad \text{ for all }1\leq j< k\leq m.
\end{equation}

Let now $p$ be a prime factor of $g_1$.
Then $p\mid \deg (a_X)$ by \eqref{formp1}, hence $p\mid g_j$ for some $j\in\set{2,\dots,m}$ by \eqref{formp}, and $p^2\mid g_1g_j\mid \deg (a_X)$ by \eqref{formp1} again.
This finishes the proof of Theorem \ref{thm:deg-chi0}.\qed

\bigskip
Given a smooth projective variety $X$ of general type, of maximal Albanese dimension, and primitive with $\chi(X,\omega_X)=0$, we know by Theorem \ref{thm:deg-chi0} that $p^2\mid \deg (a_X)$ for some prime number $p$.
Using the proofs of Theorems \ref{thm:deg-chi0} and   \ref{thm:minimal}, we study the extremal case   $\deg (a_X)=p^2$.

\begin{theo}\label{thm:extremal-chi}
Let $X$ be a smooth projective variety of general type, of maximal Albanese dimension, and $\chi(X,\omega_X)=0$.
If $\deg (a_X)=p^2$ for some prime number $p$, the morphism $a_X$ is birationally a $( \Z/p\Z)^2$-cover of its image.
\end{theo}

\begin{proof}
We use the same notation as in Theorem \ref{thm:minimal}.

There exists by Lemma \ref{lem:chi0}.4) a quotient $A_X\tto B=A_X/K$ such that the general fiber $F$ of the induced fibration $ X\tto X_B$ is primitive with $\chi(F,\omega_F)=0$.
There is a factorization
$$\res{a_X}{F} \colon F\stackrel{a_F}{\lthra} A_F\stackrel{h}{\lra} K \hookrightarrow A_X.
$$
On the other hand, we have
$$p^2=\deg(a_X)=\deg(h a_F)\deg(X_B\xrightarrow{\ a_{X_B}\ } B).$$
By Theorem \ref{thm:deg-chi0}, $\deg(a_F)$ is divisible by the square of a prime number.
It follows that this prime number must be $p$
and that $a_{X_B}$ is birational onto its image.

Let $\eta\in X_B$ be the generic point.
The geometric generic fiber $X_{\bar{\eta}}$ of $X\tto X_B$ is then primitive and satisfies $\chi(X_{\bar{\eta}},\omega_{X_{\bar{\eta}}})=0$ and $\deg(X_{\bar{\eta}}\to A_{\bar{\eta}})=p^2$.
We are therefore reduced to the case where $X$ is primitive.

Note that $a_X\colon  X\rightarrow A_X$ is {\em minimal} (see Definition \ref{def:primitive}).
We can therefore apply Theorem \ref{thm:minimal}.
Keeping its notation, we see $G$ has index $\deg(a_X)=p^2$ in $G_1\times\cdots\times G_m$; since $\pi_{ij}$ is injective whenever $1\leq i<j\leq m$, we obtain that each $G_j$ is isomorphic to $ \Z/p\Z$ and that $a_X\colon  X\rightarrow A_X$ is a $( \Z/p\Z)^2$-cover.
\end{proof}

By Proposition \ref{prop:sing-albanese}, we get the following corollary.

\begin{coro} \label{cor:extremal-rct}
Let $X$ be a projective variety of general type with rational singularities.
If $X$ is a rational cohomology torus and $\deg (a_X)=p^2$ for some prime number $p$, then $a_X$ is a $( \Z/p\Z)^2$-cover.
\end{coro}

\subsection{Simple factors of the Albanese variety}
Using the description of minimal primitive varieties of general type with $\chi=0$, we obtain restrictions on the number of simple factors of their Albanese varieties (which we already know is $\ge3$ by  Lemma \ref{lem:chi0}.1)).

\begin{prop}\label{prop:factors}
Let $X$ be a smooth projective variety of general type, of maximal Albanese dimension, with $\chi(X, \omega_X)=0$.
If $p\mid \deg (a_X)$ is the smallest prime divisor of the degree of the Albanese map $a_X$, the Albanese variety $A_X$ has at least $p+1$ simple factors.
\end{prop}

\begin{proof}
We argue by induction on $\dim (X)$.
As in Step 1 of the proof of Theorem \ref{thm:deg-chi0}, we can assume that $X$ is minimal primitive (in the sense of Definition \ref{def:primitive}).
Then, applying Theorem \ref{thm:minimal}, we may assume, after taking an \'etale cover of $X$, that $X$ is birational to a quotient $(F_1\times\cdots\times F_m)/G$, the abelian variety $A_X$ has $m$ simple factors, $V^0(\omega_X)$ has $m+1$ irreducible components: $\PB_1:=\{0_{\PK_1}\}\times \PK_2\times\cdots\times\PK_m$, $\PB_2:=\PK_1\times\{0_{\PK_2}\}\times\cdots\times\PK_m,\ldots,\PB_m:=\PK_1\times\cdots\times \PK_{m-1}\times\{0_{\PK_m}\}$, and all elements of $S_{a_X}$ contain the origin $0_{\PA_X} $.

By the Decomposition Theorem \cite[Theorem 1.1 and Theorem 3.5]{CJ}, we have
\[a_{X*}\omega_X=\bigoplus_{\PB\in S_{a_X}} p_B^*\cF_B,\]
where $p_B\colon A_X\tto B$ is the natural quotient and $\cF_B$ is a coherent sheaf supported on $B $.

On the other hand, by \cite[Lemma 3.7]{CJ}, we have, for each $j\in\set{1,\dots,m}$,
\[
a_{X_{B_j}*}\omega_{X_{B_j}}=\bigoplus_{\PB\in S_{a_X},\, \PB\subset\PB_j} p_{B,j}^*\cF_B ,
\]
where $p_{B,j}\colon B_j\tto B$ is the natural quotient, hence $\deg (a_{X_{B_j}})=1+\sum_{\PB\in S_{a_X},\,\PB\subset\PB_j}\rank( \cF_B)$.
Since all elements of $S_{a_X}$ are contained in $\bigcup_{1\leq j\leq m}\PB_j$,
we have 
\begin{equation}\label{eqn:dega}
\deg (a_X)-1\le \sum_{1\leq j\leq m}\bigl( \deg(a_{X_{B_j}})-1\bigr).
\end{equation}
With the notation of the proof of Theorem \ref{thm:minimal} ($g =|G|$ and
$g_j =|G_j|$), we get
\[\deg (a_X)=\frac{1}{g}\prod_{1\leq j\leq m}g_j\qquad\text{ and }\qquad\deg(a_{X_{B_k}})=\frac{1}{g}\prod_{j\neq k}g_j.\]
We may assume that $\deg(a_{X_{B_1}})$ is maximal among all $ \deg(a_{X_{B_k}})$.
Using \eqref{eqn:dega}, we then obtain $m\deg(a_{X_{B_1}})\ge \deg (a_X)+m-1>g_1 \deg(a_{X_{B_1}})$.
Hence $m\geq g_1+1\geq p+1$.
\end{proof}

By Proposition \ref{prop:sing-albanese}, we obtain Theorem \ref{thm:factors-rct} in the introduction as a direct corollary.

\section{Construction of examples}\label{sec:examples}
We show that the varieties in Theorem \ref{thm:extremal-chi} and Corollary \ref{cor:extremal-rct} actually occur.
The lower bounds on the degree of the Albanese morphisms in Theorem \ref{thm:deg-rct}, Theorem \ref{thm:deg-chi0} and Proposition \ref{prop:factors} are therefore optimal.

We first construct, for every prime $p$, a series of examples of minimal primitive varieties $X$ of general type with $\chi(X,\omega_X)=0$, such that $X$ is a finite quotient of a product of  $p+1$  curves and the Albanese morphism $a_X$ is a   $(\Z/p\Z )^2$-cover.
Then we show that a slight modification of this construction leads to rational cohomology tori.

\begin{exam}[Minimal primitive varieties  of general type with $\chi =0$ whose Albanese morphisms are   $(\Z/p\Z )^2 $-covers of a product of $p+1$ elliptic curves]\label{exam:chi0-p2}
Let $p$ be a prime number.
For each $j\in\set{1,\dots, p+1}$, let $\rho_j\colon  C_j\rightarrow E_j$ be a $(\Z/p\Z)$-cover, where $C_j$ is a smooth projective curve of genus $g_j\geq 2$ and $E_j$ an elliptic curve.
 For each $j\in\set{1,\dots, p+1}$, write 
\[\rho_{j*}\omega_{C_j}=\cO_{E_j}\oplus\bigoplus_{1\leq m\leq p-1}\cL_{\chi^m_j},\]
where $\chi_j\in (\Z/p\Z)^\vee$ is a generator and $\cL_{\chi^m_j}=\cL_{\chi_j}^{\otimes m}$ is an ample line bundle on $E_j$, for each   $m\in\set{1,\dots, p-1}$.
Let $\rho\colon  V=C_1\times\cdots\times C_{p+1}\rightarrow A=E_1\times\cdots \times E_{p+1}$ be the corresponding $G $-cover, where $G:= (\Z/p\Z)^{p+1}$.

We now construct a subgroup $H$ of $G$ isomorphic to $ (\Z/p\Z)^{p-1} $.
Actually, we will describe dually the quotient morphism of character groups.
Let $p_j\colon  G\rightarrow \Z/p\Z$ be the projection to the $j$-th factor, let $\chi\in (\Z/p\Z)^\vee$ be a generator, and set $\chi_j:=\chi\circ p_j\colon  G\rightarrow \C^*$.
Then $ (\chi_1,\dots,\chi_{p+1})$ is a generating family of $G^\vee\isom (\Z/p\Z)^{p+1}$.
On the other hand, let $ (e_1,\dots,e_{p-1}) $ be the canonical basis for $(\Z/p\Z)^{p-1}$.
Define the quotient morphism $G^\vee\tto H^\vee$ as follows
\begin{eqnarray*}
\pi\colon  G^\vee&\lra& (\Z/p\Z)^{p-1}\\
\chi_j&\longmapsto&
\begin{cases}
e_j& \mathrm{ for } \quad1\leq j\leq p-1;\\
\sum_{1\leq j\leq p-1}e_j& \mathrm{ for } \quad j=p;\\
\sum_{1\leq j\leq p-1}je_j & \mathrm{ for } \quad j=p+1.
\end{cases}
\end{eqnarray*}
Let $H$ be the corresponding subgroup of $G$ and set $X:=V/H$.
We consider
\[\rho\colon  V\xrightarrow{\ f\ } X\xrightarrow{\ g\ } A=E_1\times \cdots\times E_{p+1}\]
and compute
$$
\rho_*\omega_V =\ \ \bigoplus_{\tau\in G^\vee}\cL_{\tau}\ \
= \hskip-5mm \bigoplus_{(m_1,\ldots,m_{p+1})\in (\Z/p\Z)^{p+1}} (\cL_{\chi^{m_1}_1}\boxtimes\cdots \boxtimes\cL_{\chi^{m_{p+1}}_{p+1}}).
$$
Moreover, as in the proof of Proposition \ref{prop:construction1}, we have
\begin{eqnarray*}
g_*\omega_X &=&(\rho_*\omega_V)^H = \bigoplus_{\substack{\tau\in G^\vee\\\pi(\tau)=1\in H^\vee}}\cL_{\tau}\\
&=& \hskip-5mm \bigoplus_{\substack{(m_1,\ldots,m_{p+1})\in (\Z/p\Z)^{p+1}_{}\\
m_j+m_p+jm_{p+1}=0 {\rm\ for\ all\ } j\in\set{1,\dots, p-1} }} \hskip-10mm (\cL_{\chi^{m_1}_1}\boxtimes\cdots \boxtimes\cL_{\chi^{m_{p+1}}_{p+1}})\\
&=&\bigoplus_{(a, b)\in (\Z/p\Z)^2} (\cL_{\chi^{-a-b}_1} \boxtimes\cL_{\chi^{-a-2b}_2}\boxtimes\cdots \boxtimes\cL_{\chi^{-a-(p-1)b}_{p-1}}\boxtimes \cL_{\chi^a_p}\boxtimes \cL_{\chi^b_{p+1}}).
\end{eqnarray*}
For $a$ and $b$ both non-trivial, there exists a unique $j\in\set{1,\dots, p-1}$ such that $a+jb=0\in \Z/p\Z$.
Thus, $\chi(X, \omega_X)=\chi(A, g_*\omega_X)=0$.

Moreover, we see that
\begin{equation}\label{eqn:v0}
V^0(g_*\omega_X)=\bigcup_{1\leq j\leq p+1}\PE_1\times\cdots\times\{0_{\PE_j}\}\times \cdots\times \PE_{p+1}.
\end{equation}
By \cite[Theorem 1]{ch-pluri}, $X$ is of general type.

By studying the fibrations induced by the components of $V^0(g_*\omega_X)$, we deduce that $X$ is primitive by Lemma \ref{lem:chi0}.4).
Since $\deg (a_X)=p^2$, Theorem \ref{thm:deg-chi0} implies that $X$ is minimal.
\end{exam}

\begin{rema}\label{rmk:exam2}
We saw that $X$ is primitive in the sense of Definition \ref{def:primitive} (see also \cite[Section 6]{CJ}).
However, $g\colon  X\rightarrow A$ is a $(\Z/p\Z)^2$-cover.
This shows that \cite[Conjecture 6.6]{CJ} is false.
Thus, the structure of primitive varieties with $\chi=0$ is more complicated than expected.
\end{rema}

\begin{exam}[Rational cohomology tori of general type with finite quotient singularities    whose Albanese morphisms are   $(\Z/p\Z )^2$-covers of a product of $p+1$ elliptic curves]\label{exam:rt-p2}
Let $G$ be an abelian group and let $A$ be a smooth variety.
Pardini described the building data for $G$-covers from normal varieties to $A$.
A $G$-cover of $A$ corresponds to a family $(\cL_{\tau})_{\tau\in G^\vee}$ of line bundles on $A$ and effective divisors $(D_{\tau,\tau'})_{(\tau,\tau')\in G^\vee\times G^\vee}$ such that  $\cL_{\tau}\otimes\cL_{\tau'} \isom\cL_{\tau\cdot\tau'}(D_{\tau,\tau'})$  (see \cite[Theorem 2.1]{par} for more details).

 Let $p$ be a prime number, set $G:=(\Z/p\Z)^2$, and consider  the $G$-cover $g\colon  X\rightarrow A$ from Example \ref{exam:chi0-p2}.
We have $$g_*\cO_X=\bigoplus_{\tau\in G^\vee} \cL_{\tau} ^{-1} $$
and effective divisors $D_{\tau,\tau'}$ as above.

Let $P_j$ and $P_j'$ be $p$-torsion line bundles on $E_j$ such that their classes $[P_j]$ and $[P_j']$ generate the group $\PE_j[p]\isom (\Z/p\Z)^2$ of $p$-torsion line bundles on $E_j$.
Set $P:=P_1\boxtimes\cdots\boxtimes P_{p+1}$ and $P':=P_1'\boxtimes\cdots\boxtimes P_{p+1}'$.

We pick generators $\tau_1$ and $\tau_2$ of $G^\vee\isom (\Z/p\Z)^2$ and, for $(a, b)\in (\Z/p\Z)^2$, we define
\[\cL'_{\tau_1^a\cdot\tau_2^b}:=\cL_{\tau_1^a\cdot\tau_2^b}\otimes P^{\otimes a}\otimes P^{\prime\otimes b}.\]
The relations  $\cL_{\tau}\otimes\cL_{\tau'} \isom\cL_{\tau\cdot\tau'}(D_{\tau,\tau'})$  still hold for all $(\tau,\tau')\in G^\vee\times G^\vee$.
Thus, by \cite[Theorem 2.1]{par}, we get a $G$-cover $g'\colon  X'\rightarrow A$ such that
\[g_*\cO_{X'}=\bigoplus_{\tau\in G^\vee}\cL_{\tau}^{\prime-1}.\]
Both covers have the same branch divisors $D_{\tau, \tau'}$, so $X'$ has the same singularities as $X$.
Therefore, $X'$ has finite quotient singularities and we have
\[g_*'\omega_{X'_{}}=\bigoplus_{\tau\in G^\vee}\cL'_{\tau}.\]
Hence, $V^0(g_*\omega_{X'_{}})$ still generates $A$, so $X'$ is also of general type by \cite[Theorem 1]{ch-pluri}.

We saw in Example \ref{exam:chi0-p2} that for any $\tau$ non-trivial, $\cL_{\tau}$ is trivial restricted to some $E_j$.
However, since $P_j$ and $P_j'$ generate $\PE_j[p]\isom (\Z/p\Z)^2$, $\cL'_{\tau}$ restricted to $E_j$ is a non-trivial torsion line bundle.
Therefore, $H^k(A, \cL_{\tau}')=0$ for any $k\in \Z$ and $\tau$ non-trivial.
This implies $h^0(X', \omega_{X'})=1$.

Let $A'\rightarrow A$ be the abelian cover induced by $P$ and $P'$.
Since they have the same building data, the Galois covers $X\times_AA'\rightarrow A'$ and $X'\times_AA'\rightarrow A'$ are isomorphic (\cite[Theorem 2.1]{par}).
This shows that $X'\times_AA' $ is a quotient of products of curves
as in Proposition \ref{prop:construction1} and hence so is $X'$.
It follows from Proposition \ref{prop:construction1} that $X'$ is a rational cohomology torus.
\end{exam}

\medskip

We now show that rational cohomology torus of general type with mild singularities exist in any dimensions $\geq 3$.

\begin{exam}[Rational cohomology tori of general type with finite quotient singularities of any dimension $\ge 3$]\label{exam-dim}
For each $j\in\{1,2,3\}$, consider a non-zero \av\  $A_j$,
an ample line bundle $L_j$ on $A_j$, a smooth  divisor $D_j\in |2L_j|$, and a non-trivial order-$2$   line bundle $P_j$ on $A_j$.
Let $X_j\rightarrow A_j$ be the double cover associated with the data $L_j$ and $D_j$, with involutions $\tau_j$, and let $X_j'\rightarrow X_j$ be the \'etale cover defined by $P_j$, with involution $\sigma_j$.
Moreover, let $\tau_j'$ be the lifting of $\tau_j$ to $X_j'$.

Set $G:=(\Z/2\Z)^2$.
Let $H_1$, $H_2$, and $H_3$ be the three non-trivial cyclic subgroups of $G$ and let $\chi_1$, $\chi_2$, and $\chi_3$ be the three non-trivial characters of $G$, with  $\Ker(\chi_j)=H_j$.
We now define the building data to construct a $G$-cover of $A:=A_1\times A_2\times A_3$.
Let $p_j\colon  A\rightarrow A_j$ be the   projections.
We define
\begin{equation*}
\begin{aligned}
L_{\chi_1}&:=P_1\boxtimes L_2\boxtimes (L_3\otimes P_3),\\
L_{\chi_2}&:=(L_1\otimes P_1)\boxtimes P_2\boxtimes L_3,\\
L_{\chi_3}&:=L_1\boxtimes (L_2\otimes P_2)\boxtimes P_3,
\end{aligned}
\qquad\qquad
\begin{aligned}
D_{H_1}&:=p_1^*D_1,\\
D_{H_2}&:=p_2^*D_2,\\
D_{H_3}&:=p_3^*D_3.
\end{aligned}
\end{equation*}
For $\{i, j, k\}=\{1, 2, 3\}$, we have $L_{\chi_i}\otimes L_{\chi_i}\isom \cO_A(D_{H_j}+D_{H_k})$ and $L_{\chi_i}\otimes L_{\chi_j}\isom L_{\chi_k}(D_k)$.
By \cite[Theorem 2.1]{par}, there exists a $G$-cover $f\colon  X\rightarrow A$ such that 
\[f_*\cO_X=\cO_A\oplus L_{\chi_1}^{-1}\oplus L_{\chi_2}^{-1}\oplus L_{\chi_3}^{-1},\]
  with branch locus $D=D_{H_1}+D_{H_2}+D_{H_3}$.
Since $D$ is a normal crossing divisor,   a local computation shows that $X$ has finite group quotient singularities.
In particular, $X$ has rational singularities.
Actually, $X$ is isomorphic to the quotient of $X_1'\times X_2'\times X_3'$ by the automorphism group generated by $\id_1\times \tau'_2\times\sigma_3$, $\sigma_1\times\id_2\times \tau'_3$, $\tau'_1\times\sigma_2\times\id_3$, and $\tau'_1\times\tau'_2\times\tau'_3$.

As in Proposition \ref{prop:construction1}, we set $\Omega_X^s=\iota_*(\Omega_{X_{\mathrm{reg}}}^s)$, where $\iota\colon  X_{\mathrm{reg}}\hookrightarrow X$ is the   open subset of smooth points.
By \cite[Theorem 4.1]{par} or \cite[Proposition 1.2]{pardini2}, we have
\begin{equation*}
\begin{aligned}
f_*\Omega_X^s =\Omega_A^s &\oplus ({\Omega}_A^s(\mathrm{log} (D_{H_2}+D_{H_3}))\otimes L_{\chi_1}^{-1})\\
&\oplus (\Omega_A^s(\mathrm{log} (D_{H_3}+D_{H_1}))\otimes L_{\chi_2}^{-1})\\
&\oplus (\Omega_A^s(\mathrm{log} (D_{H_1}+D_{H_2}))\otimes L_{\chi_3}^{-1}),
\end{aligned}
\end{equation*}
 hence  
\[H^t(X, \Omega_X^s)\simeq H^t(A, \Omega_A^s)  \qquad \text{for all }s, t\geq 0,\] 
and $X$ is a rational cohomology torus of general type with finite quotient singularities.
Moreover, let $Y$ be the quotient of $X_1'\times X_2'\times X_3'$ by the automorphism group generated by $\id_1\times \tau'_2\times\sigma_3$, $\sigma_1\times\id_2\times \tau'_3$, and $\tau'_1\times\sigma_2\times\id_3$.
Then $Y\rightarrow X$ is a double cover and $Y$ is smooth.
\end{exam}

The following example exhibits primitive fourfolds with $\chi=0$ whose Albanese varieties have $4$ simple factors and whose Albanese morphisms have degree $8$, which is not a square.
Thus, we have essentially two different series of examples of primitive varieties with $\chi=0$ whose Albanese varieties have $4$ simple factors:
the minimal varieties provided by Example \ref{exam:chi0-p2} taking $p=3$ and the following non-minimal primitive varieties.

\begin{exam}[Non-minimal primitive fourfolds of general type with $\chi=0$ whose Albanese morphisms are   $(\Z/2\Z )^3 $-covers of a product of $4$ elliptic curves]
\label{exam:nonP2}
Let $\rho_1\colon  C_1\rightarrow E_1$ be a $(\Z/2\Z )^2$-cover, where $C_1$ is a smooth projective curve of genus $g_1\geq 2$ and $E_1$ is an elliptic curve.
Let $\sigma$ and $\tau$ be generators of the Galois group and let $\sigma^\vee,\tau^\vee\in G_1^\vee$ be the dual characters.
For $j\in\set{2, 3, 4}$, let $\rho_j\colon  C_j\rightarrow E_j$ be a $(\Z/2\Z)$-cover with associated involution $\iota_j$, where $C_j$ is a smooth projective curve of genus $g_j\geq 2$ and $E_j$ an elliptic curve.

Thus, we are considering the case where $G_1=(\Z/2\Z )^2$ and $G_2=G_3=G_4=\Z/2\Z$ in Proposition \ref{prop:construction1} or Theorem \ref{thm:minimal}.
We have
\begin{eqnarray*}
\rho_{1*}\omega_{C_1}&=&\cO_{E_1}\oplus \cL_{\sigma^\vee}\oplus \cL_{\tau^\vee}\oplus \cL_{\sigma^\vee\tau^\vee},\\
\rho_{j*}\omega_{C_j}&=&\cO_{E_j}\oplus \cL_{\iota_j^\vee} \qquad \text{for }j\in\{2,3,4\},
\end{eqnarray*}
where $\cL_{\chi^\vee}$ is an ample line bundle corresponding to the character $\chi^\vee$.

We then set $X:= (C_1\times C_2\times C_3\times C_4)\big/ \langle\sigma\times \iota_1\times \iota_2\times \mathrm{id}_{F_3}, \tau\times \iota_1\times \mathrm{id}_{F_2}\times \iota_3\rangle$ and consider the $(\Z/2\Z )^2$-quotient
\[f\colon C_1\times C_2\times C_3\times C_4 \lra X.
\]
With the notation of Theorem \ref{thm:minimal} or Proposition \ref{prop:construction1}, we have $G=(\Z/2\Z )^2$.

Let $g\colon  X\rightarrow A=E_1\times E_2\times E_3\times E_4$ be the morphism such that the composition $\rho$
\[\rho\colon  Z=C_1\times C_2 \times C_3 \times C_4\xrightarrow[\ 4:1\ ]{f} X\xrightarrow[\ 8:1\ ]{g} A,\]
is the quotient by $G_1\times G_2\times G_3\times G_4$.

Abusing the notation, we can describe the quotient ${(G_1\times G_2\times G_3\times G_4)}^\vee\to G^\vee$ by
\begin{equation*}
\begin{aligned}
&\sigma^\vee\mapsto (1,0) \\ &\tau^\vee\mapsto (0,1)
\end{aligned}
\qquad\qquad
\begin{aligned}
&\iota_1^\vee\mapsto (1,0) \\ &\iota_2^\vee\mapsto (0,1)\\ &\iota_3^\vee\mapsto(1,1).
\end{aligned}
\end{equation*}
One checks that
$$
\begin{array}{rccccl}
g_*\omega_X&\isom&(\rho_*\omega_Z)^{G} &\isom&&\cO_A\oplus ( \cO_{E_1}\boxtimes \cL_{\iota_1^\vee}\boxtimes \cL_{\iota_2^\vee}\boxtimes \cL_{\iota_3^\vee})\\
&&&&{}\oplus &( \cL_{\sigma^\vee}\boxtimes\cL_{\iota_1^\vee}\boxtimes \cO_{E_2}\boxtimes \cO_{E_3}) \oplus (\cL_{\sigma^\vee}\boxtimes \cO_{E_1}\boxtimes \cL_{\iota_2^\vee}\boxtimes \cL_{\iota_3^\vee})\\
&&&&\oplus & (\cL_{\tau^\vee}\boxtimes \cO_{E_1}\boxtimes \cL_{\iota_2^\vee}\boxtimes \cO_{E_3}) \oplus (\cL_{\tau^\vee}\boxtimes \cL_{\iota_1^\vee}\boxtimes\cO_{E_2}\boxtimes \cL_{\iota_3^\vee})\\
&&&&\oplus & (\cL_{\sigma^\vee\tau^\vee}\boxtimes \cO_{E_1}\boxtimes \cO_{E_2}\boxtimes \cL_{\iota_3^\vee})\oplus (\cL_{\sigma^\vee\tau^\vee}\boxtimes \cL_{\iota_1^\vee}\boxtimes\cL_{\iota_2^\vee}\boxtimes \cO_{E_3}).
\end{array}
$$
Thus, $\chi(X, \omega_X)=\chi(A, g_*\omega_X)=0$.
Moreover, we obtain
\begin{equation*}
V^0(g_*\omega_X)=\bigcup_{1\leq j\leq 4}\PE_1\times\cdots\times\{0_{\PE_j}\}\times \cdots\times \PE_{4}.
\end{equation*}
By \cite[Theorem 1]{ch-pluri}, $X$ is of general type and, by
studying the fibrations induced by the components of $V^0(g_*\omega_X)$, we obtain that $X$ is primitive by Lemma \ref{lem:chi0}.4).

Note that $X$ is not minimal primitive: if we consider the quotient ${(G_1\times G_2\times G_3\times G_4)}^\vee\to H^\vee\isom (\Z/2\Z )^3$ defined by
\begin{equation*}
\begin{aligned}
&\sigma^\vee\mapsto (1,0,0) \\ &\tau^\vee\mapsto (0,1,0)
\end{aligned}
\qquad\qquad
\begin{aligned}
&\iota_1^\vee\mapsto (1,0,1) \\ &\iota_2^\vee\mapsto (0,1,1)\\ &\iota_3^\vee\mapsto(1,1,1),
\end{aligned}
\end{equation*}
and the varieties $Z:=C_1\times C_2\times C_3\times C_4 $ and $Y:= Z / H$, we have
\[\rho\colon  Z \xrightarrow[\ 4:1\ ]{f} X\xrightarrow[\ 2:1\ ] {}Y \xrightarrow[\ 4:1\ ]{h} A,\]
and one checks that
$$
\begin{array}{rcccl}
h_*\omega_Y&=&(\rho_*\omega_{Z})^{H} &=&\cO_A\oplus (\cL_{\sigma^\vee}\boxtimes \cO_{E_1}\boxtimes \cL_{\iota_2^\vee}\boxtimes \cL_{\iota_3^\vee})\\
&&&&\quad {}\oplus  (\cL_{\tau^\vee}\boxtimes \cL_{\iota_1^\vee}\boxtimes\cO_{E_2}\boxtimes \cL_{\iota_3^\vee})\oplus (\cL_{\sigma^\vee\tau^\vee}\boxtimes \cL_{\iota_1^\vee}\boxtimes\cL_{\iota_2^\vee}\boxtimes \cO_{E_3}).
\end{array}
$$
Thus, $\chi(Y, \omega_Y)=\chi(A, h_*\omega_Y)=0$.
Moreover, we have $V^0(g_*\omega_X)=V^0(h_*\omega_Y)$, so by \cite[Theorem 1]{ch-pluri}, $Y$ is of general type.
\end{exam}

\addresseshere

\newpage

\pagestyle{empty}

\appendix

\section*{Appendix\\ \textbf{\uppercase{Non-existence of smooth rational cohomology tori of general type}}}

\begin{center}
\footnotesize{WILLIAM F. SAWIN}
\end{center}

\begin{abstract}
In this appendix, we resolve the smooth general type case of Catanese's question by
showing that there are no smooth general type rational cohomology tori. The key technical ingredient is a result of Popa and Schnell on one-forms on smooth general type varieties.
\end{abstract}

\begin{theoA}\label{thm:app}
Let $f \colon  X \to A$ be a finite morphism from a smooth projective variety of general type $X$ to an abelian variety $A$, all over $\C$.
Let $n$ be the dimension of $X$. Then
\[
(-1)^n {\chi}_{\rm top}(X) > 0.
\]
\end{theoA}

\begin{proof}
Recall that $(-1)^n {\chi}_{\rm top}(X)$ is the top Chern class of the cotangent bundle, or, equivalently,
the intersection number of a section of the cotangent bundle and the zero section.
We will compute this by taking a generic $1$-form of $A$ and pulling it back to $X$.
We will show that its vanishing locus is zero-dimensional and nonempty, which implies that the
intersection number is positive.

First we will show that the vanishing locus is $0$-dimensional. Let
\[Z \subseteq X \times H^0 (A, \Omega^1_A)\]
be the locus of pairs of a point $x \in X$ and a one-form $\omega$ on $A$ such that $f^*\omega$ vanishes at $x$.
Let $m$ be the dimension of $A$.
Then the dimension of $Z$ is at most $m$: because it is a closed subset, it is sufficient to check that for each subvariety $Y \subseteq X$ of dimension $k$
with generic point $\eta$, the fiber $Z_\eta$ has dimension at most $m - k$. The map $f$ remains finite
when restricted to $Y$, and finite morphisms in characteristic $0$ are generically unramified, so the map\[H^0 (A, \Omega^1_A) \otimes_{\C}  \C(\eta) = \Omega^1_{A,f(\eta)} \otimes_{\C(f(\eta)) } \C(\eta)
 \to \Omega^1_{Y,\eta}\]
 from the cotangent space of $f(\eta)$ to the cotangent space of $\eta$
is surjective, hence its kernel has dimension $m - k$.
Then the kernel of the natural map
\[H^0 (A, \Omega^1_A) \otimes_{\C}  \C(\eta) \to \Omega^1_{X,\eta}\]
has dimension at most $m - k$, because it is contained in the previous kernel.
But $Z_\eta$ is precisely the affine space corresponding to this kernel, viewed as a vector space over $\C(\eta)$. So the dimension of $Z_\eta$ equals the dimension of the kernel and is at most $m-k$, and thus the dimension of $Z$ is at most $m$, as desired. Hence the vanishing locus of a generic $1$-form from $A$ is zero-dimensional.

By a result of Popa and Schnell \cite[Conjecture 1]{popaschnell}, any $1$-form on $X$ vanishes at some point.
So the vanishing locus is nonempty.
Now the Chern number $c_n (\Omega^1_X )$ is the intersection number of the zero section with this generic $1$-form.
Because the intersection consists of finitely many points, the intersection number is a sum of contributions at those points, which is $1$ if they are transverse but is always positive in general, so the total intersection number is positive.
Thus
\[(-1)^n {\chi}_{\rm top}(X) = c_n (\Omega^1_X ) > 0.\qedhere \]
\end{proof}

\begin{corA} \label{cor:smoothrtgt}
Let $X$ be a smooth projective variety of general type. Then $X$ is not a rational cohomology torus.
\end{corA}

\begin{proof}
If it is, then by a remark of Catanese \cite[Remark 72]{Catan}, its Albanese morphism 
$ X \to A_X$
 is finite.
So by Theorem \ref{thm:app}, its topological Euler characteristic is nonzero.
But because its rational cohomology is the same as that of an abelian variety, its Euler characteristic must be the same as that of an abelian variety, which is zero.
This is a contradiction, so $X$ is not a rational cohomology torus.
\end{proof}

\addtocontents{toc}{\SkipTocEntry}
\subsection*{Acknowledgements}William F. Sawin was supported by NSF Grant No.\ DGE-1148900 when writing this appendix and would like to thank the organizers of the 2015 AMS Summer Institute on Algebraic Geometry for providing the venue where he worked on this problem.

\vspace{1cm}

\noindent \small{\textsc{Department of Mathematics, Princeton University, Fine Hall, Washington Road, NJ, USA}}\\
\small{\textit{E-mail address:} \href{mailto:wsawin@math.princeton.edu}{wsawin@math.princeton.edu}
\end{document}